\documentclass[12pt]{article}

\usepackage{graphics,amsmath,amssymb,amsthm,mathrsfs}

\newcommand{\br}{\mathbb{R}}

\newcommand{\varep}{\varepsilon}

\newtheorem{thm}{Theorem}[section]
\newtheorem{lemma}[thm]{Lemma}

\newtheorem{prop}[thm]{Proposition}

\newtheorem{definition}[thm]{Definition}
\newtheorem{remark}[thm]{Remark}

\numberwithin{equation}{section}

\begin{document}

\bibliographystyle{amsplain}

\title{Layer Potential Methods \\
for Elliptic Homogenization Problems}

\author{Carlos E. Kenig\thanks{Supported in part by NSF grant DMS-0456583}
 \and Zhongwei Shen\thanks{Supported in part by NSF grant DMS-0855294}}

\date{ }

\maketitle

\begin{abstract}

In this paper we use the method of layer potentials to study $L^2$
boundary value problems in a bounded Lipschitz domain $\Omega$
for a family of second order elliptic systems with rapidly oscillating periodic
coefficients, arising in the
theory of homogenization.
Let $\mathcal{L}_\varep=-\text{div}\big(A(\varep^{-1}X)\nabla \big)$.
Under the assumption that $A(X)$ is elliptic, symmetric, periodic and
H\"older continuous, we establish the solvability of the
$L^2$ Dirichlet, regularity, and Neumann problems for $\mathcal{L}_\varep
(u_\varep)=0$ in $\Omega$ with optimal
estimates uniform in $\varep>0$.

\end{abstract}

\section{Introduction}

This paper continues the study in \cite{kenig-shen-1}
of elliptic homogenization problems in 
Lipschitz domains. 
Let $\Omega$ be a bounded Lipschitz domain in $\mathbb{R}^{d}$, $d\ge 3$.
Consider a family of
second order elliptic systems $\mathcal{L}_\varepsilon (u_\varepsilon) =0$
in $\Omega$, where $u_\varepsilon
=(u_\varepsilon^1, \dots, u_\varepsilon^m)$ and
\begin{equation}
\mathcal{L}_\varepsilon
=-\frac{\partial}{\partial x_i}
\left[ a_{ij}^{\alpha\beta} \left(\frac{X}{\varepsilon}\right) \frac{\partial }{\partial x_j}\right]
=-\text{div}\left[ A\left(\frac{X}{\varepsilon}\right)\nabla \right],\quad
\quad \varepsilon>0.
\end{equation}
We will assume that the coefficient matrix 
$A(X)=(a_{ij}^{\alpha\beta} (X))$ 
is real and
satisfies the ellipticity condition, 
\begin{equation}\label{ellipticity}
 \mu |\xi|^2 \le a_{ij}^{\alpha\beta} (X) \xi_i^\alpha \xi_j^\beta \le \frac{1}{\mu} |\xi|^2 
\quad \text{ for }X\in \mathbb{R}^{d} \text{ and }  \xi=(\xi_i^\alpha)\in \mathbb{R}^{dm},
\end{equation}
where $\mu>0$,
and the periodicity condition,
\begin{equation}\label{periodicity}
A(X+Z)=A(X) \quad \text{ for } X\in \mathbb{R}^{d} \text{ and }
Z\in \mathbb{Z}^{d}.
\end{equation}
We shall also impose the smoothness condition,
\begin{equation}\label{smoothness}
| A(X)-A(Y)| \le \tau|X-Y|^\lambda
\quad \text{ for some } \lambda\in (0,1) \text{ and } \tau\ge 0,
\end{equation}
and the symmetry condition $A^*=A$, i.e.,
\begin{equation}\label{symmetry}
a_{ij}^{\alpha\beta} (X) =a_{ji}^{\beta\alpha} (X)  \quad \text{ for } 1\le i,j\le d
\text{ and } 1\le \alpha, \beta \le m.
\end{equation}
Under these conditions, we establish the solvability of
 the $L^2$ Dirichlet, regularity and Neumann problems
for $\mathcal{L}_\varep (u_\varep)=0$ in $\Omega$ with optimal estimates
that are uniform in the parameter $\varep>0$.

We say $A\in \Lambda (\mu, \lambda,\tau)$ if it satisfies conditions 
(\ref{ellipticity}), (\ref{periodicity}) and (\ref{smoothness}). 
The following are the main results of this paper.

\begin{thm}\label{e-Dirichlet-theorem}
Let $\Omega$ be a bounded Lipschitz domain in $\mathbb{R}^d$, $d\ge 3$
with connected boundary.
Let $\mathcal{L}_\varep =-\text{div}\big(A(\varep^{-1}X)\nabla \big)$
with $A\in \Lambda (\mu,\lambda,\tau)$ and $A^*=A$.
Then for any $f\in L^2(\partial\Omega, \mathbb{R}^m)$, 
there exists a unique $u_\varep$ such that
$\mathcal{L}_\varep (u_\varep)=0$ in $\Omega$,
$(u_\varep)^*\in L^2(\partial\Omega)$ and
$u_\varep=f$ n.t. on $\partial\Omega$.
Moreover, the solution $u_\varep$ satisfies the
estimate $\|(u_\varep)^*\|_2 \le C\| f\|_2$ with 
constant $C$ independent of $\varep>0$.
Furthermore, $u_\varep$ may be represented by a double layer potential
with density $g_\varep\in L^2(\partial\Omega, \mathbb{R}^m)$ and
$\| g_\varep\|_2 \le C\| f\|_2$.
\end{thm}

Here $(u_\varep)^*$ denotes the usual nontangential maximal function of $u_\varep$ and
$\|\cdot\|_p$ the norm in $L^p(\partial\Omega)$.
By $u_\varep=f$ n.t. on $\partial\Omega$, 
we mean that $u$ converges to $f$ nontangentially.

\begin{thm}\label{e-regularity-theorem}
Suppose that $\Omega$ and $A$ satisfy the same conditions
as in Theorem \ref{e-Dirichlet-theorem}. Then
for any $f\in W^{1,2}(\partial\Omega, \mathbb{R}^m)$, there exists
a unique $u_\varep$ such that
$\mathcal{L}_\varep (u_\varep)=0$ in $\Omega$,
$(\nabla u_\varep)^*\in L^2 (\partial\Omega)$ and
$u_\varep=f $ n.t. on $\partial\Omega$.
Moreover, the solution $u_\varep$ satisfies the
estimate $\|(\nabla u_\varep)^*\|_2 \le C \|\nabla_{tan} f \|_2$
with $C$ independent of $\varep>0$.
Furthermore, $\nabla u_\varep$ exists n.t. on $\partial\Omega$ and
$u_\varep$ may be represented by a single layer potential
with density $g_\varep\in L^2(\partial\Omega, \mathbb{R}^m)$
and 
$
\| g_\varep \|_2 \le C \{ \|\nabla_{tan} f\|_2 
+[\sigma(\partial\Omega)]^{\frac{1}{1-d}} \|f\|_2\}
$.
\end{thm}

Let $\frac{\partial u}{\partial\nu_\varep}$ denote the conormal
derivative associated with the operator $\mathcal{L}_\varep$.
We will use $L^p_0(\partial\Omega, \mathbb{R}^m)$
to denote the subspace of functions in $L^p(\partial\Omega, \mathbb{R}^m)$
with mean value zero.

\begin{thm}\label{e-Neumann-theorem}
Suppose that $\Omega$ and $A$ satisfy the same conditions 
as in Theorem \ref{e-Dirichlet-theorem}.
Then for any $f\in L_0^2(\partial\Omega, \mathbb{R}^m)$,
there exists a $u_\varep$, unique up to constants, such that
$\mathcal{L}_\varep (u_\varep)=0$ in $\Omega$,
$(\nabla u_\varep )^*\in L^2 (\partial\Omega)$ and
$\frac{\partial u_\varep}{\partial\nu_\varep} =f$ 
n.t. on $\partial\Omega$.
Moreover, the solution $u_\varep$ satisfies the estimate
$\|(\nabla u_\varep)^*\|_2 \le C\| f\|_2$ with $C$ independent of $\varep>0$.
Furthermore, $\nabla u_\varep$ exists n.t. on $\partial\Omega$
and $u_\varep$ may be represented 
by a single layer potential with density $g_\varep\in L_0^2 (\partial\Omega,
\mathbb{R}^m)$ and $\| g_\varep\|_2 \le C \| f\|_2$.
\end{thm}

A few remarks are in order.

\begin{remark}\label{remark-1.1}
{\rm 
In the case of $m=1$, the $L^p$  Dirichlet problem  
for $\mathcal{L}_\varep (u_\varep)=0$ in Lipschitz domains
with uniform estimate $\|(u_\varep)^*\|_p\le C\| u_\varep\|_p$
was solved for $2-\delta<p<\infty$
by Dahlberg \cite{dahlberg-personal}, who extended an earlier work
of Avellaneda and Lin \cite{AL-1987-ho} for domains satisfying the uniform
exterior ball condition.
Recently the authors initiated the study of
the $L^p$ Neumann and regularity problems for 
$\mathcal{L}_\varep (u_\varep)=0$
with uniform estimates on $\|(\nabla u_\varep)^*\|_p$ in \cite{kenig-shen-1} .
Under the assumption that $A$ is elliptic, symmetric, periodic and satisfies
a certain square-Dini condition, we solve the $L^p$ Neumann and regularity problems
in Lipschitz domains for the sharp range $1<p<2+\delta$.
A new proof of Dahlberg's theorem on the Dirichlet problem
is also given in \cite{kenig-shen-1}. 
We mention that $L^p$ Neumann and regularity problems for 
a general second order elliptic equation
were formulated and studied in \cite{KP-1993, KP-1995} (see \cite{Kenig-1994}
for references on related work on the $L^p$ boundary value problems with minimal
smoothness assumptions). 
}
\end{remark}

\begin{remark}\label{remark-layer-potential}
{\rm
Theorems \ref{e-Dirichlet-theorem}, \ref{e-regularity-theorem} and \ref{e-Neumann-theorem}
extend the analogous results for the second order elliptic systems with constant
coefficients satisfying (\ref{ellipticity})
in Lipschitz domains \cite{Verchota-1984,
dahlberg3, fabes2, fabes, Gao-1991}
(in the constant coefficient case, some results
also hold when (\ref{ellipticity}) is relaxed to the so-called
Legendre-Hadamard ellipticity condition; see \cite{dahlberg3, Verchota-1986, Gao-1991}).
As in the case of elliptic systems with constant coefficients,
our results for elliptic systems with periodic coefficients
 are established by the method of layer potentials -the classical method
of integral equations.
We point out that the use of layer potentials in the periodic setting relies on
two crucial developments.
The first one is the proof of Coifman-McIntosh-Meyer \cite{coifman}
of the $L^p$ boundedness of the Cauchy integrals on Lipschitz curves.
By the method of rotation this gives the $L^p$ boundedness of layer
potentials on Lipschitz surfaces
for elliptic systems with constants coefficients.
Using the method of freezing coefficients,
one also obtains the $L^p$ boundedness of layer potentials
on Lipschitz surfaces, in small scales, 
for systems with variable coefficients
(see e.g. \cite{Mitrea-Taylor-1999, Mitrea-Taylor-2000}
on the use of layer potentials in the study of boundary
value problems for the Laplace-Beltrami operator on Lipschitz
sub-domains of Riemannian manifolds).
However, to treat Lipschitz surfaces for large scales,
for operators with periodic coefficients, 
we need to appeal to the work of
Avellaneda and Lin on elliptic homogenization problems.
In a series of remarkable papers \cite{AL-1987, AL-1987-ho, 
AL-1989-ho, AL-1989-II, AL-1991}, Avellaneda and Lin
established the uniform $C^{0,\alpha}$, $C^{0,1}$, $L^p$ estimates
for the family of operators $\{ \mathcal{L}_\varep\}$
in smooth domains.
In \cite{AL-1991} they also obtained certain decay estimates 
for large scales on the matrix
of fundamental solutions of $\mathcal{L}_\varep$.
It is these estimates that enable us to show that 
$\|(\nabla u_\varep)^*\|_p \le C \| g\|_p$, if $u_\varep$ is given by
the single layer potential for the operator $\mathcal{L}_\varep$
with density $g$ (see Sections 2 and 3).
}
\end{remark}

\begin{remark}\label{remark-rescaling}
{\rm
Note that the estimates in terms of nontangential maximal functions,
\begin{equation}\label{nontangential-estimate}
\|(u_\varep)^*\|_2 \le C\| u_\varep \|_2,\quad
\|(\nabla u_\varep )^*\|_2 \le C\| \nabla_{tan} u_\varep\|_2,\quad
\|(\nabla u_\varep )^*\|_2 \le C \|\frac{\partial u_\varep }{\partial\nu_\varep }\|_2
\end{equation}
in Theorems \ref{e-Dirichlet-theorem},
\ref{e-regularity-theorem} and \ref{e-Neumann-theorem} are scale-invariant.
Thus by a simple rescaling argument,
 one may reduce the proof of Theorems \ref{e-Dirichlet-theorem},
\ref{e-regularity-theorem} and \ref{e-Neumann-theorem}
to the case $\varep=1$, provided that in this special case one can show that
the constant $C$ in (\ref{nontangential-estimate}) depends only
on $d$, $m$, $\mu$, $\lambda$, $\tau$ and the Lipschitz character of $\Omega$.
}
\end{remark}

With the availability of the method of layer potentials and
following the approach for elliptic systems with constant coefficients in Lipschitz domains,
we are led to the Rellich estimates $\|\nabla u\|_2
\le C\| \nabla_{tan} u\|_2 $ and $\|\nabla u\|_2
\le C \|\frac{\partial u}{\partial \nu}\|_2$
for suitable solutions of $\mathcal{L}(u)=0$,
where $\mathcal{L}=\mathcal{L}_1$.
Let $\psi:\mathbb{R}^{d-1}\to \mathbb{R}$ be a Lipschitz function
such that $\psi(0)=0$ and $\|\nabla\psi\|_\infty\le M$.
By localization techniques we may further reduce the problem to proving
 the following {\it a-priori} estimates for solutions of $\mathcal{L}(u)=0$ in $D(2r)$,
\begin{equation}\label{Rellich-estimate-1.1}
\aligned
&\int_{\Delta(r)} |\nabla u|^2\, d\sigma
\le C\int_{\Delta (2r)} \big|\frac{\partial u}{\partial \nu}\big|^2\, d\sigma
+\frac{C}{r}
\int_{D(2r)} |\nabla u|^2\, dX,\\
&\int_{\Delta(r)} |\nabla u|^2\, d\sigma
\le C\int_{\Delta (2r)} |\nabla_{tan} u |^2\, d\sigma
+\frac{C}{r}
\int_{D(2r)} |\nabla u|^2\, dX,
\endaligned
\end{equation}
where $\Delta(r)=\{ (x^\prime, \psi(x^\prime))\in \mathbb{R}^d: |x^\prime|<r\}$,
$ D(r)=\{ (x^\prime, x_d): |x^\prime|<r \text{ and }
\psi (x^\prime)<x_d< 10\sqrt{d}(M +1)r\}$, and
the constant $C$ depends only on $d$, $m$, $\mu$, $\lambda$, 
$\tau$ and $M$.

The proof of (\ref{Rellich-estimate-1.1}) is divided into two parts.
Part one deals with the small-scale case $0<r\le 1$.
We mention that in this case, if $A\in C^1(\mathbb{R}^d)$, the desired estimates
follow readily from the same Rellich-Nec\u{a}s-Payne-Weinberger formulas as in the case of
constant coefficients, with constant $C$ depending on $\|\nabla A\|_\infty$.
Our proof of (\ref{Rellich-estimate-1.1}) in small scales for H\"older continuous
coefficients, which also uses the Rellich formulas,
involves a delicate three-step approximation argument (see Sections 6 and 7).
This is needed to obtain the correct dependence 
on the constant $C$ in (\ref{Rellich-estimate-1.1}), in contrast to arguments in \cite{Mitrea-Taylor-2000}
where the dependence of the constants are not clear.
Part two, which is given in Section 8, treats the large-scale case $r>1$
and uses the periodicity assumption on $A$.
Here we first apply the small-scale estimates and reduce the problem to the control
of the integral of $|\nabla u|^2$ on a boundary layer $\{(x^\prime, x_d):
|x^\prime|<r \text{ and } \psi(x^\prime)< x_d< \psi(x^\prime)+1\}$.
The desired estimates of the integral over the boundary layer
follow from certain integral identities we developed in \cite{kenig-shen-1}
for elliptic operators with periodic coefficients.
These identities may be regarded as Rellich type identities
for operators with $x_d$-periodic coefficients.
We point out that in the case of constant coefficients, Rellich identities
are usually derived by using integration by parts
on some forms involving $\frac{\partial u}{\partial x_d}$.
The basic insight here is to replace the $x_d$
derivative of $u$ by the difference $Q(u)(x^\prime, x_d)
=u(x^\prime, x_d+1)-u(x^\prime, x_d)$.
The periodicity of $A$ is used in the fact that
$Q(u)$ is a solution whenever $u$ is a solution.
In \cite{kenig-shen-1}
the approach outlined above was used to solve the $L^p$
boundary value problems for elliptic equations
with periodic coefficients 
by the method of $\mathcal{L}$-harmonic measures
(see Remark \ref{remark-1.1}).
With the method
of layer potentials, the approach works equally well
for elliptic systems with periodic
coefficients, at least in the case $p=2$.
It is worth mentioning that the symmetry assumption (\ref{symmetry}),
although not needed for the uniform boundedness of layer potentials, 
is essential for the Rellich estimates both in small and large scales,
which are needed for the uniform invertibility.

By the stability of 
Fredholm properties of operators on a complex
interpolation scale, the $L^2$ results in Theorems \ref{e-Dirichlet-theorem},
\ref{e-regularity-theorem} and \ref{e-Neumann-theorem}
extend easily to the $L^p$ setting for $p\in (2-\delta, 2+\delta)$,
where $\delta>0$ depends only on $d$, $m$, $\mu$, $\lambda$, $\tau$
and the Lipschitz character of $\Omega$.
Using the $L^p$ techniques developed in \cite{Dahlberg-Kenig-1990,
Shen-2006-ne, Shen-2007-boundary} for constant coefficients,
we may further extend the results for the Dirichlet problem 
with $L^p$ boundary data to the range
$2< p\le\infty$ if $d=3$, and to $2<p<\frac{2(d-1)}{d-3}+\delta$ if $d\ge 4$.
Similarly, the $L^p$ Neumann and regularity problems for
$\mathcal{L}_\varep (u_\varep)=0$
may be solved for $1<p< 2$ if $d=3$, and for $\frac{2(d-1)}{d+1}-\delta
<p<2$ if $d\ge 4$.
These results, as well as uniform Sobolev and Besov estimates
for $\mathcal{L}_\varep$ in nonsmooth domains
(see \cite{Shen-2008} for uniform $W^{1,p}$ estimates in the case $m=1$)
will appear elsewhere.

We end this section with a few notations and definitions that will be used throughout
the paper.

For a ball $B=B(X,r)$ in $\mathbb{R}^d$ with center $X$ and
radius $r$, we will denote $B(X,tr)$ by $tB$.
If $0<\lambda<1$, then
$
\| f\|_{C^{0,\lambda}(\Omega)} = 
\inf \{M: |f(X)-f(Y)|\le M|X-Y|^\lambda \text{ for }X,Y\in \Omega\}$.
We will let
$\| f\|_{C^\lambda (\mathbb{R}^d)}=\| f\|_{L^\infty (\mathbb{R}^d)} +\| f \|_{C^{0, \lambda}(\mathbb{R}^d)}$.

Let $\Omega$ be a bounded Lipschtz domain.
We say $\Omega\in \mathbf{\Pi} (M, N)$ for some $M>0$ and $N>10$, if
there exist $r>0$ and $\{ P_i: i=1,\dots, N\}\subset \partial\Omega$
such that $\partial\Omega \subset \bigcup_i B(P_i, r)$ and for each $i$,
there exists a coordinate system, obtained from the standard Euclidean
system through translation and rotation, so that $P_i=(0,0)$ and
$$
B(P_i,C_M r)\cap\Omega =B(P_i,C_M r)\cap \big\{
(x^\prime, x_d)\in \mathbb{R}^d:\ x^\prime\in \mathbb{R}^{d-1}
\text{ and }
x_d>\psi_i (x^\prime)\big\},
$$
where $C_M =10(M+1)$,
 $\psi_i: \mathbb{R}^{d-1}\to \mathbb{R}$ is a Lipschitz function,
$\psi_i (0)=0$ and $\|\nabla\psi_i\|_\infty\le M$.
Note that if $\Omega\in \mathbf{\Pi} (M,N)$, then its dilation
$\varep \Omega =\{ \varep X: \ X\in \Omega\} \in \mathbf{\Pi} (M, N)$
for any $\varep>0$.
A constant $C$ is said to depend on the Lipschitz character of $\Omega$
if there exist $M$ and $N$ such that $\Omega \in \mathbf{\Pi}(M,N)$ and
the constant can be made uniform for any Lipschitz domain in $\mathbf{\Pi} (M,N)$.
We will call $C$ a ``good'' constant if it depends at most
on $d$, $m$, $\mu$ in (\ref{ellipticity}), $\lambda$ and $\tau$ in 
(\ref{smoothness}), and the Lipschitz character of $\Omega$.

Finally the summation convention will be used throughout this paper.

\section{Matrix of fundamental solutions}

Let $\mathcal{L}=\mathcal{L}^A =-\text{div}(A(X)\nabla)$.
Under the ellipticity condition (\ref{ellipticity}) and H\"older condition
(\ref{smoothness}) on $A(X)$,
it is well known that the gradients of weak solutions to
$\mathcal{L}(u)=\text{div}(f)$ are locally H\"older continuous, provided
that $f$ is H\"older continuous.
 More precisely,
let $B=B(X_0, R)$ for some $X_0\in \mathbb{R}^d$ and $0<R\le 1$.
There exists $C=C(d,m,\mu, \lambda, \tau)>0$ such that
if $u\in W^{1,2}(2B)$ is a weak solution to $\mathcal{L}(u)=\text{div}(f)$ in $2B$, then 
\begin{equation}\label{local-Holder-gradient-estimate}
 \|\nabla u\|_{C^{0, \lambda}(B)}
 \le \frac{C}{R^{1+\lambda}} 
\left\{\frac{1}{|2B|}
\int_{2B} |u|^2\, dX\right\}^{1/2}
+C\| f\|_{C^{0, \lambda} (2B)}
\end{equation}
(see e.g. \cite{Giaquinta}, p.88). Also, the gradient $\nabla u$ is locally bounded and
\begin{equation}\label{local-gradient-estimate}
\|\nabla u\|_{L^\infty(B)}
 \le \frac{C}{R} 
\left\{\frac{1}{|2B|}
\int_{2B} |u|^2\, dX\right\}^{1/2}
+C R^\lambda\| f\|_{C^{0, \lambda} (2B)}.
\end{equation}
If $R>1$, estimates (\ref{local-Holder-gradient-estimate})-(\ref{local-gradient-estimate})
 still hold. However the constant $C$ may depend on $R$.
With the additional periodicity condition (\ref{periodicity}), 
Avellaneda and Lin, among other things,
 were able to establish the following global
gradient estimate in \cite{AL-1987} (p. 826).

\begin{lemma}\label{gradient-estimate-lemma}
Let $B=B(X_0,R)$ for some $X_0\in \mathbb{R}^d$ and $R>0$.
Suppose that 
 $u\in W^{1,2}(2B)$ is a weak solution 
to $\text{div}(A\nabla u)=0$ in $2B$ for some $A\in \Lambda(\mu, \lambda, \tau)$.
Then 
\begin{equation}\label{gradient-estimate}
\sup_{B} |\nabla u|
\le\frac{C}{R}
\left\{ \frac{1}{|2B|}\int_{2B} |u|^2\, dX\right\}^{1/2},\end{equation}
where $C$ depends only on $d$, $m$, $\mu$, $\lambda$ and $\tau$.
\end{lemma}

Using Lemma \ref{gradient-estimate-lemma}, one can construct a matrix-valued
function $\Gamma (X,Y)=(\Gamma^{\alpha\beta} (X,Y))_{m\times m}$   such that
for each $Y\in \br^d$, 
$\nabla_X \Gamma (X,Y)$ is locally integrable and
\begin{equation}\label{fundamental-solution-representation}
\phi^\gamma (Y)
=\int_{\br^d} a_{ij}^{\alpha\beta} (X) \frac{\partial}{\partial x_j}
\Gamma^{\beta\gamma} (X,Y) \frac{\partial }{\partial x_i} \phi^\alpha (X)\, dX
\end{equation}
for $\phi=(\phi^1, \dots, \phi^m)\in C_0^\infty(\br^d, \br^m)$.
Moreover, $\Gamma(X,Y)$ satisfies
the estimates
\begin{equation}\label{size-estimate}
\aligned
|\Gamma(X,Y)| & \le C|X-Y|^{2-d},\\
|\nabla_X \Gamma(X,Y)|+ |\nabla_Y \Gamma(X,Y)| &\le C|X-Y|^{1-d},
\endaligned
\end{equation}
where $C$ depends only on $d$, $m$, $\mu$, $\lambda$ and $\tau$
(see e.g. \cite{Hofmann-Kim}).
The function $\Gamma(X,Y)=\Gamma_A (X,Y)$ is
called the matrix of fundamental solutions
for the operator $\mathcal{L}$ in $\mathbb{R}^d$, with pole at $Y$. 
Note that 
\begin{equation}\label{adjoint-relation}
\big(\Gamma_A (X,Y)\big)^* =\Gamma_{A^*} (Y,X),
\end{equation}
where $A^*$ denotes the adjoint matrix of $A$.
Since $\nabla_Y \Gamma (\cdot, Y)$ is a weak solution in $\br^d\setminus \{ Y\}$,
we also have
\begin{equation}\label{2-derivative-estimate}
|\nabla_X\nabla_Y \Gamma(X,Y)|\le C|X-Y|^{-d}\quad
\text{ for any }X, Y\in \mathbb{R}^d, X\ne Y.
\end{equation}

If $E$ is a real constant matrix satisfying (\ref{ellipticity}), we will let
$\Theta (X,Y;E)$ to denote $\Gamma_E (X,Y)$, the matrix of fundamental solutions
for the operator $-\text{div}(E\nabla)$. 
Note that
$\Theta (X,Y;E)=\Theta(X-Y,0;E)=\Theta(Y-X,0; E)=\Theta(Y,X;E)$.
Also, $\Theta (X,0;E)$ is homogeneous of degree $2-d$ in $X$ and
\begin{equation}\label{constant-coefficient-estimate}
|\nabla^N_X \Theta (X,0;E)|\le C|X|^{2-d-N}
\end{equation}
for any $N\ge 0$, 
where $C$ depends only on $d$, $m$, $\mu$ and $N$
(see \cite{Morrey-1954}). Moreover, if $E, \widetilde{E}$
are two constant matrices satisfying (\ref{ellipticity}), 
then
\begin{equation}\label{constant-coefficient-difference-estimate}
|\nabla^N_X \Theta (X,0; E)-\nabla_X^N \Theta(X,0;\widetilde{E})
\le C \|E-\widetilde{E}\|\,  |X|^{2-d-N},
\end{equation}
where $C=C(d,m,\mu,N)>0$.
We remark that estimate (\ref{constant-coefficient-difference-estimate})
may be proved by an argument similar to that in the proof of
Lemma \ref{fundamental-difference-lemma}.

For a function $F=F(X,Y,Z)$, we will use the notation
 $$
\nabla_1 F(X,Y,Z)=\nabla_X F(X,Y,Z) \text{ and } \nabla_2 F(X,Y,Z)=\nabla_Y F(X,Y,Z).
$$
The following lemma describes the local behavior of $\Gamma_A (X,Y)$.

\begin{lemma}\label{local-behavior-lemma}
Let $A\in \Lambda (\mu, \lambda, \tau)$.
Then for
 any $X,Y\in \mathbb{R}^{d}$,
\begin{equation}\label{local-behavior-estimate-1}
\aligned
|\Gamma_A (X,Y)-\Theta (X, Y;A(X))| &\le C |X-Y|^{2-d+\lambda},\\
|\nabla_X \Gamma_A (X,Y)-\nabla_1 \Theta (X, Y; A(X))|
& \le C|X-Y|^{1-d+\lambda},\\
|\nabla_X \Gamma_A (X,Y)-\nabla_1 \Theta (X, Y; A(Y))|
& \le C|X-Y|^{1-d+\lambda},\\
\endaligned
\end{equation}
where $C>0$ depends only on $d$, $m$, $\mu$, $\lambda$ and $\tau$.
\end{lemma}

\begin{proof}
Let $\widetilde{A}\in \Lambda (\mu,\lambda,\tau)$.
Then
\begin{equation}\label{fundamental-solution-difference}
{\Gamma}_{\widetilde{A}}^{\alpha\delta}(X,Y)-\Gamma_A^{\alpha\delta} (X,Y)
=\int_{\br^d}
\frac{\partial}{\partial z_i}
\Gamma^{\alpha\beta}_A(X,Z) \big\{
a_{ij}^{\beta\gamma} (Z)-\widetilde{a}_{ij}^{\beta\gamma} (Z)\big\}
\frac{\partial}{\partial z_j} {\Gamma}^{\gamma\delta}_{\widetilde{A}} (Z,Y)\, dZ
\end{equation}
for any $X,Y\in \br^d$ (see e.g. \cite{Hofmann-Kim}).
It follows from (\ref{fundamental-solution-difference}) 
and estimates (\ref{size-estimate}) and (\ref{2-derivative-estimate}) that
\begin{equation}\label{difference-estimate-1}
|{\Gamma}_{\widetilde{A}} (X,Y)-\Gamma_A (X,Y)|\le
C\int_{\br^d} \frac{|A(Z)-\widetilde{A}(Z)|}{|Z-X|^{d-1} |Z-Y|^{d-1}}\, dZ
\end{equation}
and
\begin{equation}\label{difference-estimate-2}
|\nabla_1{\Gamma}_{\widetilde{A}} (X,Y)-\nabla_1 \Gamma_A (X,Y)|\le
C\int_{\br^d} \frac{|A(Z)-\widetilde{A}(Z)|}{|Z-X|^{d} |Z-Y|^{d-1}}\, dZ.
\end{equation}

To show the first inequality in (\ref{local-behavior-estimate-1}),
we fix $X\in \mathbb{R}^d$ and let $\widetilde{A}=A(X)$. Then $\Gamma_{\widetilde{A}}
(X,Y)=\Theta(X,Y;A(X))$ and by (\ref{difference-estimate-1}),
$$
\aligned
|\Gamma_A (X,Y)-\Theta (X,Y;A(X))|
& \le C\int_{\mathbb{R}^d} \frac{|A(Z)-A(X)|}{|Z-X|^{d-1}|Z-Y|^{d-1}}\, dZ\\
& \le C\int_{\mathbb{R}^d} \frac{dZ}{|Z-X|^{d-1-\lambda}{|Z-Y|^{d-1}}}\\
&\le C|X-Y|^{2-d+\lambda}.
\endaligned
$$
The second inequality in (\ref{local-behavior-estimate-1})
follows from (\ref{difference-estimate-2})
in the same manner.
Note that by (\ref{constant-coefficient-difference-estimate}),
\begin{equation}\label{constant-coefficient-estimate-3}
|\nabla_1 \Theta (X,Y;A(X))-\nabla_1 \Theta (X,Y; A(Y))|\le C|X-Y|^{1-d+\lambda}.
\end{equation}
The third inequality in (\ref{local-behavior-estimate-1})
follows from the second and (\ref{constant-coefficient-estimate-3}).
\end{proof}

\begin{remark}{\rm
If we fix $Y\in \mathbb{R}^d$ and let $\widetilde{A}=A(Y)$, the same argument as in the proof of 
Lemma \ref{local-behavior-lemma} yields that
\begin{equation}\label{local-behavior-estimate-2}
\aligned
|\Gamma_A (X,Y)-\Theta (X,Y;A(Y))| &\le C|X-Y|^{2-d+\lambda},\\
|\nabla_Y \Gamma_A  (X,Y) -\nabla_2 \Theta (X,Y;A(Y))|
&\le C|X-Y|^{1-d+\lambda},\\
|\nabla_Y \Gamma_A  (X,Y) -\nabla_2 \Theta (X,Y;A(X))|
&\le C|X-Y|^{1-d+\lambda}.
\endaligned
\end{equation}
}
\end{remark}

To study the behavior of $\Gamma_A (X,Y)$ for $|X-Y|\ge 1$, 
we need to introduce the matrix of correctors,
$\chi =\chi (X)=(\chi_{i}^{\alpha\beta} (X))$, $1\le i\le d$, $1\le \alpha, \beta\le m$.
Here for each $i$ and $\alpha$,
$\chi_i^\alpha =(\chi_{i}^{\alpha 1}, \dots, \chi_{i}^{\alpha m})$ 
is the solution of the following cell problem:
\begin{equation}\label{corector}
\left\{
\aligned
& \mathcal{L}(\chi_{i}^\alpha)=\mathcal{L}(e^\alpha x_i) \quad \text{ in } \mathbb{R}^{d},\\
& \chi_{i}^\alpha \text{ is periodic with respect to }\mathbb{Z}^{d},\\
& \int_{[0,1]^{d}} \chi_{i}^{\alpha}\, dX=0,
\endaligned
\right.
\end{equation}
where $e^\alpha =(0, \dots, 1, \dots, 0)\in \mathbb{R}^m$ with $1$ in the $\alpha^{th}$ position.
Note by estimate (\ref{gradient-estimate}), $\|\nabla \chi\|_\infty\le C$
for some $C=C(d,m,\mu,\lambda, \tau)$.
Let $\mathcal{L}_0=-\text{div}(A_0\nabla )$ denote the homogenized
elliptic operator associated with $\{ \mathcal{L}_\varep\}$,
where $A_0$ is a constant matrix in $\Lambda (\mu, \lambda, \tau)$
(see e.g. \cite{Bensoussan-1978}, p.121
for the explicit formula of $A_0$, given in terms of $a_{ij}^{\alpha\beta} (X)$
and $\chi_i^{\alpha\beta}(X))$.

The following lemma was proved
in \cite{AL-1991}.

\begin{lemma}\label{global-behavior-lemma}
Let $A\in \Lambda(\mu, \lambda, \tau)$. Then
\begin{equation}\label{global-behavior-estimate}
\aligned
|\Gamma^{\alpha\beta}_A (X,Y)-\Gamma^{\alpha\beta}_{A_0} (X,Y)| 
&\le C|X-Y|^{2-d-\lambda_0},\\
|\frac{\partial}{\partial x_i}\Gamma^{\alpha\beta}_A(X,Y)-
\frac{\partial}{\partial x_i}\Gamma^{\alpha\beta}_{A_0}(X,Y)-
\frac{\partial }{\partial x_i} \chi_j^{\alpha\gamma} (X)\cdot \frac{\partial}{\partial x_j}
\Gamma_{A_0}^{\gamma\beta} (X,Y)|
& \le C |X-Y|^{1-d-\lambda_0},
\endaligned
\end{equation}
for any $X,Y\in \mathbb{R}^d$,
where $C>0$ and $\lambda_0\in (0,1)$ depend only on $d$, $m$, $\mu$, $\lambda$ and $\tau$.
\end{lemma}

\begin{remark}
{\rm Let $I$ denote the identity matrix (or the identity operator). 
For brevity the second estimate in (\ref{global-behavior-estimate}) may be written as
\begin{equation}\label{global-behavior-estimate-1}
|\nabla_X \Gamma_A (X,Y) -(I +\nabla \chi (X) )\nabla_X \Gamma_{A_0} (X,Y)|
\le C |X-Y|^{1-d-\lambda_0}.
\end{equation}
Using (\ref{adjoint-relation}), one may also deduce that
\begin{equation}\label{global-behavior-estimate-2}
|\nabla_Y (\Gamma_A (X,Y))^*
-(I+\nabla \chi^* (Y)) \nabla_Y (\Gamma_{A_0} (X,Y))^*|\le
C|X-Y|^{1-d-\lambda_0},
\end{equation}
where $\chi^*$ is the matrix of correctors for the adjoint operator $\mathcal{L}^* =
-\text{div}(A^*(X)\nabla)$.
}
\end{remark}

The rest of this section is devoted to the estimate of
$\Gamma_A (X,Y)-\Gamma_{\widetilde{A}} (X,Y)$ 
and its derivatives
when
$A$ is close to $\widetilde{A}$ in the space
${C^\lambda}(\mathbb{R}^d)$.
The results on the derivative estimates
are local and will be used in an approximation
argument for domains $\Omega$ with diam$(\Omega)\le 1$.

\begin{lemma}\label{fundamental-difference-lemma}
Let $A,\widetilde{A}\in \Lambda (\mu, \lambda, \tau)$.
Then for any $X,Y\in \mathbb{R}^d$,
\begin{equation}\label{fundamental-difference-estimate-1}
|\Gamma_A (X,Y)-\Gamma_{\widetilde{A}} (X,Y)|
\le C \| A-\widetilde{A}\|_\infty |X-Y|^{2-d},
\end{equation}
where $C=C(d, m, \mu, \lambda, \tau)>0$.
Moreover, for each $R\ge 1$, there exists  a constant $C_R$ depending on
$d$, $m$, $\mu$, $\lambda$, $\tau$ and $R$ such that
\begin{equation}\label{fundamental-difference-estimate-2}
\aligned
|\nabla_ X \Gamma_A (X,Y)-\nabla_X \Gamma_{\widetilde{A}} (X,Y)|
&\le C_R \| A-\widetilde{A}\|_{C^\lambda (\mathbb{R}^d)} |X-Y|^{1-d},\\
|\nabla_X\nabla_Y\Gamma_A (X,Y)-\nabla_X\nabla_Y\Gamma_{\widetilde{A}} (X,Y)|
&\le C_R \| A-\widetilde{A}\|_{C^\lambda(\mathbb{R}^d)} |X-Y|^{-d},
\endaligned
\end{equation}
for any $X,Y\in \mathbb{R}^d$ with $|X-Y|\le R$.
\end{lemma}

\begin{proof}
It follows from estimate (\ref{difference-estimate-1}) that
$$
\aligned
|\Gamma_A (X,Y)-\Gamma_{\widetilde{A}}(X,Y)|
& \le C\| {A}-\widetilde{A}\|_\infty
\int_{\mathbb{R}^d}
\frac{dZ}{|X-Z|^{d-1}|Z-Y|^{d-1}}\\
& \le C\| {A}-\widetilde{A}\|_\infty |X-Y|^{2-d}.
\endaligned
$$

To see (\ref{fundamental-difference-estimate-2}), we fix $X_0,Y_0\in \mathbb{R}^d$ and
consider $u(X)=\Gamma_A (X, Y_0)-\Gamma_{\widetilde{A}} (X,Y_0)$
in $2B$, where $B=B(X_0,r/4)$ and $r=|X_0-Y_0|\le R$.
It follows from (\ref{fundamental-difference-estimate-1}) that
$\|u\|_{L^\infty(2B)} \le Cr^{2-d}\| A-\widetilde{A}\|_\infty$.
Let $w(X)=\Gamma_{\widetilde{A}}(X,Y_0)$.
By (\ref{size-estimate}), $\|\nabla w\|_{L^\infty(2B)} \le Cr^{1-d}$.
In view of (\ref{local-Holder-gradient-estimate}), we also have
$\| \nabla w\|_{C^{0, \lambda}(2B)}\le C_R r^{1-d-\lambda}$.
Since
$
\mathcal{L}^A( u)=\text{div}(A\nabla w)=\text{div}\big((A-\widetilde{A})\nabla w\big)$
in $2B$,
it follows from (\ref{local-gradient-estimate}) that
$$
\aligned
\| \nabla u\|_{L^\infty(B)}
&\le C r^{-1} \|u\|_{L^\infty(2B)}
+C r^{\lambda} \| (A-\widetilde{A})\nabla w\|_{C^{0,\lambda} (2B)}\\
& \le
C r^{1-d} \| A-\widetilde{A}\|_\infty
+Cr^\lambda \| A-\widetilde{A}\|_{C^\lambda(\mathbb{R}^d)}
\| \nabla w\|_{C^\lambda(2B)}\\
&\le C r^{1-d} \| A-\widetilde{A}\|_{C^\lambda(\mathbb{R}^d)},
\endaligned
$$
where $C$ may depend on $R$. 
This gives the first inequality in (\ref{fundamental-difference-estimate-2}).
The second inequality in (\ref{fundamental-difference-estimate-2})
follows in the same manner.
Indeed, let $v(X)=\nabla_Y \Gamma_A (X,Y_0)-\nabla_Y \Gamma_{\widetilde{A}}(X, Y_0)$
and $g(X)=\nabla_Y \Gamma_{\widetilde{A}} (X,Y_0)$.
Then $\mathcal{L}^A(v)=\text{div}\big( (A-\widetilde{A})\nabla g\big)$ in $2B$.
Thus,
$$
\|\nabla v\|_{L^\infty(B)}
\le Cr^{-1} \| v\|_{L^\infty(2B)}
+Cr^\lambda \| A-\widetilde{A}\|_{C^\lambda(\mathbb{R}^d)}
\|\nabla g\|_{C^\lambda(2B)}.
$$
It follows from (\ref{adjoint-relation}) and
the first inequality in (\ref{fundamental-difference-estimate-2}) 
that
$\|v\|_{L^\infty(2B)}
\le Cr^{1-d} \| A-\widetilde{A}\|_{C^\lambda(\mathbb{R}^d)}$.
By (\ref{local-Holder-gradient-estimate})
and (\ref{size-estimate}), we see that $\| \nabla g\|_{C^\lambda(2B)} \le Cr^{-d-\lambda}$.
Hence 
$$
\|\nabla v\|_{L^\infty(B)}
\le Cr^{-d} \| A-\widetilde{A}\|_{C^\lambda(\mathbb{R}^d)},
$$
where $C$ may depend on $R$. This completes the proof.
\end{proof}

Define
\begin{equation}\label{kernel-difference}
\Pi_A (X,Y)
=\nabla_X \Gamma_A (X,Y)-\nabla_1 \Theta (X,Y;A(X)).
\end{equation}

\begin{lemma}\label{kernel-estimate-lemma}
Let $A, \widetilde{A}\in
\Lambda (\mu, \lambda, \tau)$ and $R\ge 1$.
Then for $X,Y$ with $|X-Y|\le R$,
\begin{equation}\label{kernel-estimate-1}
|\Pi_A (X,Y)-\Pi_{\widetilde{A}} (X,Y)|
\le C_R \| A-\widetilde{A}\|_{C^\lambda(\mathbb{R}^d)} 
|X-Y|^{1-d+\lambda},
\end{equation}
 where $C_R =C(d,m,\mu,\lambda,\tau, R)>0$.
\end{lemma}

\begin{proof}
Let $\Omega=B(X_0,2R)$.
For any fixed $P\in B(X_0,R)$, it follows from integration by parts that
\begin{equation}\label{local-difference-formula}
\aligned
& \Gamma_A^{\alpha\delta}(X,Y)- \Theta^{\alpha\delta}
(X,Y;A(P))\\
&=\int_\Omega \frac{\partial}{\partial z_i} \Gamma_A^{\alpha\beta}(X,Z)
\big\{ a_{ij}^{\beta\gamma} (P)-a_{ij}^{\beta\gamma}(Z)\big\}
\frac{\partial}{\partial z_j} \Theta^{\delta\gamma}(Z,Y; A(P))\, dZ\\
&\qquad
+\int_{\partial\Omega} 
\frac{\partial}{\partial z_i} \Gamma_A^{\alpha\beta}(X,Z)
a_{ij}^{\beta\gamma}(Z) n_j(Z) \Theta^{\gamma\delta}(Z,Y; A(P))\, d\sigma(Z)\\
&\qquad
-\int_{\partial\Omega} \Gamma_A^{\alpha\beta}(X,Z)
n_i(Z) a_{ij}^{\beta\gamma}(P) \frac{\partial}{\partial z_j}
\Theta^{\gamma\delta}(Z,Y; A(P))\, d\sigma(Z),
\endaligned
\end{equation}
where $n=(n_1, \dots, n_d)$ denotes the unit outward normal to $\partial\Omega$.
We now take the derivative with respect to $X$ on the both side of
(\ref{local-difference-formula}) and then choose $X=P$. With abuse of notation we may write
\begin{equation}
\aligned
\Pi_A (X,Y)
= &\int_\Omega
\nabla_X \nabla_Z \Gamma_A (X,Z)
\big\{ A(X)-A(Z)\big\} \nabla_1 \Theta (Z,Y;A(X))\, dZ\\
& \qquad +\int_{\partial\Omega} \nabla_Z\nabla_X \Gamma_A (X,Z)
A(Z) n(Z) \Theta (Z,Y;A(X))\, d\sigma(Z)\\
& \qquad-\int_{\partial\Omega}
\nabla_X \Gamma_A (X,Z) n(Z) A(X) \nabla_Z \Theta (Z,Y; A(X))\, d\sigma (Z).
\endaligned
\end{equation}
To estimate
$
\Pi_A (X,Y)-\Pi_{\widetilde{A}} (X,Y)
$,
we split its solid integrals as
$I_1 +I_2 +I_3$, where
$$
\aligned
I_1
&=\int_\Omega
\big\{ \nabla_X\nabla_Z \Gamma_A (X,Z)-\nabla_X\nabla_Z \Gamma_{\widetilde{A}}
(X,Z)\big\} 
\big\{ A(X)-A(Z)\big\} \nabla_1 \Theta (Z,Y;A(X))\, dZ,\\
I_2
&=\int_\Omega
\big\{\nabla_X\nabla_Z \Gamma_{\widetilde{A}}
(X,Z)\big\} 
\big\{ A(X)-A(Z)- (\widetilde{A}(X)-\widetilde{A}(Z))
\big\} \nabla_1 \Theta (Z,Y;A(X))\, dZ,\\
I_3
&=\int_\Omega
\big\{\nabla_X\nabla_Z \Gamma_{\widetilde{A}}
(X,Z)\big\} 
\big\{ \widetilde{A}(X)-\widetilde{A}(Z)\big\} 
\big\{ \nabla_1 \Theta (Z,Y;A(X))
-\nabla_1 \Theta (Z,Y; \widetilde{A}(X))\big\}\, dZ.
\endaligned
$$
It follows from (\ref{fundamental-difference-estimate-2}) that
$$
\aligned
|I_1|& \le C_R\|A-\widetilde{A}\|_{C^\lambda(\mathbb{R}^d)}
\int_\Omega \frac{dZ}{|X-Z|^{d-\lambda}|Z-Y|^{d-1}}\\
& \le C_R \| A-\widetilde{A}\|_{C^{\lambda}(\mathbb{R}^d)} |X-Y|^{1-d+\lambda}.
\endaligned
$$
Similarly, by estimates (\ref{2-derivative-estimate}) 
and (\ref{constant-coefficient-difference-estimate}),
$$
\aligned
& |I_2|
\le C\| A-\widetilde{A}\|_{C^{0, \lambda}(\mathbb{R}^d)} |X-Y|^{1-d+\lambda},\\
&|I_3|\le C\| A-\widetilde{A}\|_\infty |X-Y|^{1-d+\lambda}.
\endaligned
$$

Finally we may split the surface integrals in $\Pi_A (X,Y)-\Pi_{\widetilde{A}}(X,Y)$
in a similar fashion to show that they are bounded by 
$C_R \| A-\widetilde{A}\|_{C^\lambda (\mathbb{R}^d)}$.
\end{proof}

\begin{remark}\label{kernel-remark}
{\rm 
Let 
\begin{equation}
\Delta_A (X,Y)=\nabla_Y \Gamma_A (X,Y)-\nabla_2 \Theta (X,Y;A(Y)).
\end{equation}
Let $A,\widetilde{A} \in \Lambda (\mu, \lambda, \tau)$
and $R\ge 1$. Using (\ref{adjoint-relation}), one may deduce
from (\ref{kernel-estimate-1}) that
for $X,Y\in \mathbb{R}^d$ with $|X-Y|\le R$,
\begin{equation}\label{kernel-estimate-2}
|\Delta_A (X,Y)-\Delta_{\widetilde{A}} (X,Y)|
\le C_R \|A-\widetilde{A}\|_{C^\lambda(\mathbb{R}^d)} |X-Y|^{1-d+\lambda},
\end{equation}
where $C_R$ depends only on $d$, $m$, $\mu$, $\lambda$, $\tau$ and $R$.
}
\end{remark}

\section{Singular integral operators on Lipschitz surfaces}

Let $\Omega$ be a bounded Lipschitz domain in $\mathbb{R}^d$.
Consider two singular integral operators on $\partial\Omega$,
\begin{equation}\label{definition-of-T}
\aligned
T_A^1(f)(P) & =\text{\rm p.v.}\int_{\partial\Omega}
\nabla_1 \Gamma_A(P,Y) f(Y)\, d\sigma(Y)\\
& :=\lim_{\rho\to 0^+}
\int_{Y\in \partial\Omega,\, |Y-P|>\rho} 
\nabla_1 \Gamma_A (P,Y) f(Y)\, d\sigma(Y),\\
T_A^2(f)(P) & =\text{\rm p.v.}\int_{\partial\Omega}
\nabla_2 \Gamma_A(P,Y) f(Y)\, d\sigma(Y),
\endaligned
\end{equation}
where $A\in  \Lambda(\mu, \lambda,\tau)$.
We also introduce the maximal singular integral operators
$T_A^{1,*}$ and $T_A^{2,*}$ on $\partial\Omega$, defined by
\begin{equation}
\aligned
T_A^{1,*}(f)(P)
&=\sup_{\rho>0}
\big| \int_{Y\in \partial\Omega,\,  |Y-P|>\rho} 
\nabla_1 \Gamma_A (P,Y) f(Y)\, d\sigma(Y)\big|,\\
T_A^{2,*}(f)(P)
&=\sup_{\rho>0}
\big| \int_{Y\in \partial\Omega,\, |Y-P|>\rho} 
\nabla_2 \Gamma_A (P,Y) f(Y)\, d\sigma(Y)\big|.
\endaligned
\end{equation}

The main purpose of this section is to establish the following.

\begin{thm}\label{operator-boundedness-theorem}
Let $f\in L^p(\partial\Omega)$ for some $1<p<\infty$.
Then $T^1_A(f)(P)$ and $T^2_A(f)(P)$ exist for a.e. $P\in \partial\Omega$
and
$$
\| T^{1,*}_A (f)\|_p +\|T^{2,*}_A (f)\|_p
\le C_p \| f\|_p,
$$
where $C_p $ depends only on $d$, $m$, $\mu$, $\lambda$, $\tau$, $p$ and
the Lipschitz character of $\Omega$.
\end{thm}

Let $\psi:\mathbb{R}^{d-1}\to \mathbb{R}$ be a Lipschitz function with $\|\nabla\psi\|_\infty
\le M$ and 
\begin{equation}\label{Lipschitz-region}
D =\big\{ (x^\prime,x_d): x^\prime\in \mathbb{R}^{d-1} \text{ and } x_d> \psi (x^\prime)\big\}.
\end{equation}
By a partition of unity and rotation of coordinate systems, 
it suffices to prove Theorem \ref{operator-boundedness-theorem}
when $\Omega=D$, with constant $C_p$ depending
only on $d$, $m$, $\mu$, $\lambda$, $\tau$, $p$ and $M$.
To this end, the basic idea to treat $T^{1,*}_A$
 is to approximate the integral kernel $\nabla_1\Gamma_A (P,Y)$
by $\nabla_1 \Theta (P,Y;A(P))$ when $|P-Y|\le 1$, and
by $(I+\nabla\chi (P))\nabla_1 \Gamma_{A_0} (P,Y)$ when $|P-Y|\ge 1$.
The operator $T^{2,*}_A$ may be handled in a similar manner.

Let $\mathcal{M}_{\partial D}$ denote the Hardy-Littlewood maximal operator on
$\partial D$.
The proof of Theorem \ref{operator-boundedness-theorem}
relies on the following two lemmas.

\begin{lemma}\label{approximation-lemma-1}
Let $\Omega=D$ be given by (\ref{Lipschitz-region}). 
Then for each $P\in \partial D$,
$$
\aligned
T_A^{1,*}(f)(P)
& \le C\mathcal{M}_{\partial D}(f) (P) +2\sup_{\rho>0} \big| \int_{Y\in \partial D,\, |Y-P|>\rho}
\nabla_1 \Theta (P,Y;A(P)) f(Y)\, d\sigma (Y)\big|\\
&\qquad + C \sup_{\rho>0}
\big|\int_{Y\in \partial D,\, |Y-P|>\rho}
\nabla_1 \Gamma_{A_0}(P,Y) f(Y)\, d\sigma (Y)\big|,\\
T_A^{2,*}(f)(P)
& \le C\mathcal{M}_{\partial D}(f) (P) +2\sup_{\rho>0} \big| \int_{Y\in \partial D, |Y-P|>\rho}
\nabla_2 \Theta (P,Y;A(Y)) f(Y)\, d\sigma (Y)\big|\\
&\qquad +\sup_{\rho>0}
\big|\int_{Y\in \partial D,\, |Y-P|>\rho}
\nabla_2 \Gamma_{A_0}(P,Y) g(Y)\, d\sigma (Y)\big|,
\endaligned
$$
where $C=C(d,m,\mu, \lambda, \tau, M)$
and $|g|\le C|f|$ on $\partial D$.
\end{lemma}

\begin{proof}
Fix $P\in \partial D$ and $\rho>0$. If $\rho\ge 1$, we use estimate
(\ref{global-behavior-estimate-1}) to obtain
$$
\aligned
& \big| \int_{|Y-P|>\rho} \nabla_1 \Gamma_A (P,Y) f(Y)\, d\sigma (Y)\big|\\
& \le C \big|\int_{|Y-P|>\rho}
\nabla_1 \Gamma_{A_0}(P,Y) f(Y)\, d\sigma (Y)\big|
+C\int_{|Y-P|>\rho} |Y-P|^{1-d-\lambda_0} |f(Y)|\, d\sigma (Y)\\
&
\le C \sup_{t>0}\big|\int_{|Y-P|>t}
\nabla_1 \Gamma_{A_0}(P,Y) f(Y)\, d\sigma (Y)\big|
+ C \mathcal{M}_{\partial D} (f) (P).
\endaligned
$$
If $0<\rho<1$, we write $\{ |Y-P|>\rho\}$ as
$\{ |Y-P|>1\} \cup \{ 1\ge |Y-P|>\rho\}$.
The integral of $\nabla_1 \Gamma_A (P,Y) f(Y)$ on $\{|Y-P|>1\} $ may be treated as above.
To handle the integral on $\{ 1\ge |Y-P|>\rho\}$,
we use estimate (\ref{local-behavior-estimate-1}) to obtain
$$
\aligned
&\big|\int_{1\ge |Y-P|>\rho}
\nabla_1 \Gamma_A (P,Y) f(Y)\, d\sigma(Y)\big|\\
&
\le \big|\int_{1\ge |Y-P|>\rho}
\nabla_1 \Theta (P,Y; A(P)) f(Y)\, d\sigma (Y)\big|
+C\int_{|Y-P|\le 1}
|P-Y|^{1-d+\lambda} |f(Y)|\, d\sigma(Y)\\
&\le 2\sup_{t>0}
\big|\int_{|Y-P|>t}
\nabla_1 \Theta (P,Y; A(P)) f(Y)\, d\sigma (Y)\big|
+C \mathcal{M}_{\partial D} (f) (P).
\endaligned
$$
This gives the desired estimate for $T^{1,*}_A$.
The estimate for $T^{2,*}_A$ follows from (\ref{global-behavior-estimate-2})
and (\ref{local-behavior-estimate-2})
in the same manner.
\end{proof}

\begin{lemma}\label{local-operator-boundedness-lemma}
Let $K(X,Y)$ be odd in $X$ and homogeneous of degree $1-d$ in $X$.
Assume that for all $0\le N\le N(d)$ where $N(d)$ is sufficiently large,
$\nabla^N_X K(X,Y)$ is continuous on $\mathbb{S}^{d-1}\times \mathbb{R}^d$
and $|\nabla^N_X K(X,Y)|\le C_0$ for $X\in \mathbb{S}^{d-1}$ and $Y\in \mathbb{R}^d$.
Let $f\in L^p(\partial D)$ for some $1<p<\infty$. Define
$$
\aligned
S^1(f)(P) & =\text{\rm p.v.}\int_{\partial D} K(P-Y, P) f(Y)\, d\sigma (Y),\\
S^2(f)(P) & =\text{\rm p.v.}\int_{\partial D} K(P-Y, Y) f(Y)\, d\sigma (Y),\\
S^{1,*}(f)(P) &= \sup_{\rho>0}\big|
\int_{Y\in \partial D,\, |Y-P|>\rho} K(P-Y, P) f(Y)\, d\sigma (Y)\big|,\\
S^{2,*}(f)(P) &= \sup_{\rho>0}\big|
\int_{Y\in \partial D, \, |Y-P|>\rho} K(P-Y, Y) f(Y)\, d\sigma (Y)\big|.
\endaligned
$$
Then $S^1(f)(P)$ and $S^2(f)(P)$ exist for a.e. $P\in \partial D$ and
$$
\| S^{1,*}(f)\|_p +\| S^{2,*} (f)\|_p \le C C_0 \| f\|_p,
$$
where $C$ depends only on $d$, $p$ and $M$.
\end{lemma}

\begin{proof} By considering $C_0^{-1} K(X,Y)$, we may clearly assume that $C_0=1$.
In the special case where the integral
kernel $K(X,Y)$ is independent of $Y$, the result is a consequence of 
 \cite{coifman} on Cauchy integrals on Lipschitz curves.
The general case may be deduced from the special case by
the spherical harmonic decomposition (see e.g. \cite{Mitrea-Taylor-1999}).
Note that only the continuity condition in the variable $Y$ is need for
$\nabla_X^N K(X,Y)$.
\end{proof}

We are now in a position to give the proof of Theorem \ref{operator-boundedness-theorem}.

\noindent{\bf Proof of Theorem \ref{operator-boundedness-theorem}.}

If $A_0\in \Lambda(\mu,\lambda,\tau)$ is a constant matrix,
the boundedness of $T^{1,*}_{A_0}$ and $T^{2,*}_{A_0}$ on $L^p(\partial D)$
for $1<p<\infty$ follows from \cite{coifman} and is well known. 
Thus, in view of Lemma \ref{approximation-lemma-1},
we only need to treat the maximal singular integral operators with kernels
$\nabla_1 \Theta (P,Y; A(P))$ and $\nabla_2\Theta(P,Y; A(Y))$.

Let $K_A (X,Y)=\nabla_1 \Theta (X,0; A(Y))$. We may write
\begin{equation}\label{relation-1}
 \nabla_1 \Theta (P,Y; A(P))
=\nabla_1 \Theta (P-Y,0; A(P))=K_A (P-Y,P).
\end{equation}
Similarly, we have
\begin{equation}\label{relation-2}
\nabla_2 \Theta (P,Y; A(Y))
=\nabla_1 \Theta (Y-P,0; A(Y))
=K_A (Y-P,Y).
\end{equation}
Recall that $\Theta (X,0;A(Y))$
is the matrix of fundamental solutions for the constant coefficient operator
$L^Y=-\text{div}(A(Y)\nabla)$ with pole at the origin. It follows that
$K_A (X,Y)$ is odd in $X$ and homogeneous of degree $1-d$ in $X$. Moreover, for any $N\ge 1$,
$\nabla_X^N K_A (X,Y)$ is continuous on $(\mathbb{R}^d\setminus \{ 0\})\times \mathbb{R}^d$ and
\begin{equation}
|\nabla_X^N  K_A (X,Y)|\le C |X|^{1-d-N}
\end{equation}
where $C=C(d,m, \mu, N)>0$.
In view of (\ref{relation-1})-(\ref{relation-2}) and Lemma \ref{local-operator-boundedness-lemma},
we may conclude that the maximal singular integral operators with
kernels $\nabla_1 \Theta (P,Y;A(P))$ and $\nabla_2 \Theta (P,Y; A(Y))$
are bounded on $L^p(\partial D)$ and their operator norms are bounded by
$C(d,m, \mu, \lambda, p, M)$.

Finally we note that for $f$ with compact support,
the existence of $T_A^1(f)(P)$ and $T_A^2(f) (P)$  
for a.e. $P\in \partial D$ follows readily 
from estimates (\ref{local-behavior-estimate-1} ) and (\ref{local-behavior-estimate-2})
 and Lemma \ref{local-operator-boundedness-lemma}.
The general case follows from this and the boundedness of
$T_A^{1,*}$ and $T_A^{2,*}$ on $L^p(\partial D)$.
\qed
\bigskip

The following theorem will be useful to us in an approximation argument.

\begin{thm}\label{operator-approximation-theorem}
Let $\Omega$ be a bounded Lipschitz domain.
Let $T_A^1, \ T_{\widetilde{A}}^1, \
T^2_A,\ T_{\widetilde{A}}^2$ be defined by
(\ref{definition-of-T}), where $A,\widetilde{A}\in \Lambda(\mu, \lambda,\tau)$.
Suppose that diam$(\Omega)\le R$ for some $R\ge 1$.
Then for $1<p<\infty$,
\begin{equation}
\aligned
\| T_A^1 (f)-T_{\widetilde{A}}^1 (f)\|_p  & \le C_R \| A-\widetilde{A}\|_{C^{\lambda}(\mathbb{R}^d)}
\| f\|_p,\\
\| T_A^2 (f)-T_{\widetilde{A}}^2 (f)\|_p  & \le C_R \| A-\widetilde{A}\|_{C^{\lambda}(\mathbb{R}^d)}
\| f\|_p,\\
\endaligned
\end{equation}
where $C_R $ depends only on $d$, $m$, $\mu$, $\lambda$, $\tau$, $p$, the Lipschitz character
of $\Omega$ and $R$.
\end{thm}

\begin{proof}
Recall that $\Pi_A (P,Y)=\nabla_1 \Gamma_A (P,Y)-\nabla_1 \Theta (P,Y; A(P))$.
Then
\begin{equation}
\aligned
&\nabla_1 \Gamma_A (P,Y)-\nabla_1 \Gamma_{\widetilde{A}} (P,Y)\\
& =\Pi_A (P,Y)-\Pi_{\widetilde{A}} (P,Y)
+\big\{ \nabla_1 \Theta (P-Y,0; A(P))
-\nabla_1 \Theta (P-Y,0; \widetilde{A}(P))\big\}.
\endaligned
\end{equation}
It follows from estimate (\ref{kernel-estimate-1}) 
that the norm of the integral operator with kernel
$\Pi_A (P,Y)-\Pi_{\widetilde{A}} (P,Y)$ on $L^p(\partial \Omega)$ is 
bounded by $ C_R \| A-\widetilde{A}\|_{C^\lambda (\mathbb{R}^d)}$
for $1\le p\le \infty$.
Note that
\begin{equation}
\big| \nabla_Z^N \big\{
\Theta (Z,0;A(P))-\Theta (Z,0; \widetilde{A}(P)\big\}\big|
\le C \| A-\widetilde{A}\|_\infty |Z|^{2-d-N}
\end{equation}
for any $N\ge 0$ and any $Z\in \mathbb{R}^d$, where $C$ depends only
on $d, m, \mu$ and $N$.
We may deduce from Lemma \ref{local-operator-boundedness-lemma} that
the norm of the integral operator with kernel
$\nabla_1 \Theta (P-Y,0; A(P))-\nabla_1 \Theta (P-Y, 0; \widetilde{A}(P))$
on $L^p(\partial\Omega)$ for $1<p<\infty$
is bounded by $C \| A-\widetilde{A}\|_\infty$.
This gives the desired estimate for $\| T_A^1(f)-T_{\widetilde{A}}^1 (f)\|_p$.
The estimate for $\| T_A^2 (f)-T^2_{\widetilde{A}} (f)\|_p$
follows from Remark \ref{kernel-remark} and Lemma \ref{local-operator-boundedness-lemma}
in the same manner.
\end{proof}

We end this section with a theorem on the nontangential maximal functions.
For $f\in L^p(\partial \Omega)$, consider the following two functions
\begin{equation}\label{function-u-w}
\aligned
u(X) & =\int_{\partial \Omega}
\nabla_X \Gamma_A (X,Y) f(Y)\, d\sigma (Y),\\
w(X)&=\int_{\partial \Omega}
\nabla_Y \Gamma_A (X,Y) f(Y)\, d\sigma (Y),
\endaligned
\end{equation}
defined on $\mathbb{R}^{d}\setminus \partial \Omega$.

\begin{thm}\label{maximal-function-theorem}
Let $\Omega$ be a bounded Lipschitz domain.
Let $u$ and $w$ be defined by (\ref{function-u-w}), where
$A\in \Lambda(\mu, \lambda,\tau)$.
Then for $1<p<\infty$,
$\| (u)^*\|_p +\| (w)^*\|_p \le C_p \| f\|_p$, where $(\cdot)^*$ denotes the nontangential maximal
function taken with respect to $\Omega_+=\Omega$ or $\Omega_-
=\mathbb{R}^d\setminus \overline{\Omega}$, and
$C_p$ depends only on $d$, $m$, $\mu$, $\lambda$, $\tau$, $p$ and
the Lipschitz character of $\Omega$.
\end{thm}

\begin{proof}
Let $\mathcal{M}_{\partial\Omega}$ denote the usual
Hardy-Littlewood maximal operator on $\partial\Omega$.
We claim that for any $P\in \partial \Omega$,
\begin{equation}\label{claim-1}
\aligned
(u)^* (P)
& \le C \mathcal{M}_{\partial \Omega} (f) (P)
+C\sup_{\rho>0}
\big|\int_{Y\in \partial\Omega,\, |Y-P|>\rho}
\nabla_1 \Theta (P,Y;A(P)) f(Y)\, d\sigma (Y)\big|\\
&\qquad\qquad\qquad
+C\sup_{\rho>0}
\big|
\int_{Y\in \partial\Omega;\, |Y-P|>\rho}
\nabla_1 \Gamma_{A_0} (P,Y) f(Y)\, d\sigma (Y)\big|,\\
(w)^* (P)
& \le C \mathcal{M}_{\partial \Omega} (f) (P)
+C\sup_{\rho>0}
\big|\int_{Y\in \partial\Omega,\, |Y-P|>\rho} 
\nabla_2 \Theta (P,Y;A(Y)) f(Y)\, d\sigma (Y)\big|\\
&\qquad\qquad\qquad
+C\sup_{\rho>0}
\big|
\int_{Y\in \partial\Omega,\, |Y-P|>\rho}
\nabla_2 \Gamma_{A_0} (P,Y) g(Y)\, d\sigma (Y)\big|,
\endaligned
\end{equation}
where $C$ is a ``good'' constant and $|g|\le C|f|$ on $\partial\Omega$.
Estimate $\|(u)^*\|_p +\| (w)^*\|_p\le C_p \| f\|_p$ follows from 
(\ref{claim-1}), as in the proof of Theorem \ref{operator-boundedness-theorem}.
We will give the proof for $(u)^*$. The estimate for $(w)^*$ may be carried out
in the same manner.

Fix $P\in \partial\Omega$. Let $X\in \mathbb{R}^d\setminus \partial\Omega$
such that $|X-P|<C_0 \text{dist} (X, \partial\Omega)$.
Let $r=|X-P|$.
If $r\ge 1$, we may use (\ref{global-behavior-estimate-1}) to show that
$$
|u(X)-(I+\nabla \chi (X)) \nabla U(X)|
\le C\int_{\partial \Omega} \frac{|f(Y)|\, d\sigma(Y)}{|X-Y|^{d-1+\lambda_0}}
\le C \mathcal{M}_{\partial \Omega} (f) (P),
$$
where $U(X)=\int_{\partial \Omega}  \Gamma_{A_0} (X,Y) f(Y)\, d\sigma(Y)$.
It follows that well known estimates for $\nabla U$ that
$$
\aligned
|u(X)| & \le C \mathcal{M}_{\partial \Omega}(f)(P) + C(\nabla U)^*(P)\\
 &\le C \mathcal{M}_{\partial \Omega}(f)(P) +C\sup_{\rho>0}
\big| \int_{|Y-P|>\rho}
\nabla_1 \Gamma_{A_0} (P,Y) f(Y)\, d\sigma(Y)\big|.
\endaligned
$$

Next suppose that $r=|X-P|< 1$. 
We write $u(X)=J_1 +J_2 +J_3$, where $J_1$, $J_2$, $J_3$ denote
the integrals of $\nabla_1\Gamma(X,Y) f(Y)$ over 
$E_1=\{ Y\in \partial \Omega: |Y-P|<r\}$,
$E_2=\{ Y\in\partial \Omega: r\le |Y-P|\le 1\}$,
$E_3=\{ Y\in \partial \Omega:  |Y-P|>1\}$,
respectively.
Clearly, $|J_1|\le C \mathcal{M}_{\partial \Omega} (f)(P)$.
For $J_2$, we use (\ref{local-behavior-estimate-1}) to obtain
$$
\aligned
|J_2| & \le \big|\int_{E_2} \nabla_1 \Theta(X,Y; A(Y)) f(Y)\, d\sigma (Y)\big|
+C\int_{E_2} \frac{|f(Y)|d\sigma (Y)}{|X-Y|^{d-1-\lambda}}\\
& \le \big| \int_{E_2}
\nabla_1 \Theta(P,Y;A(Y)) f(Y)\, d\sigma(Y)\big|
+C\int_{E_2}
\frac{|f(Y)|\, d\sigma (Y)}{|Y-P|^{d-1-\lambda}}\\
&\le 
2\sup_{\rho>0} \big| \int_{|Y-P|>\rho}
\nabla_1 \Theta(P,Y;A(Y)) f(Y)\, d\sigma (Y)\big|
+C \mathcal{M}_{\partial D} (f)(P).
\endaligned
$$
In view of (\ref{global-behavior-estimate-1}), we have
$$
\aligned
|J_3|
& \le C\int_{E_3} \frac{|f(Y)|\, d\sigma(Y)}{|X-Y|^{d-1+\lambda_0}}
+C\big|\int_{E_3}
\nabla_1 \Gamma_{A_0} (X,Y) f(Y)\, d\sigma (Y)\big|\\
&\le C \mathcal{M}_{\partial D} (f)(P)
+C\big|\int_{E_3} \nabla_1 \Gamma_{A_0} (P,Y) f(Y)\, d\sigma (Y)\big|.
\endaligned
$$
This, together with estimates of $J_1$ and $J_2$, yields
the desired estimate for $(u)^*(P)$.
\end{proof}

\section{Method of layer potentials}

In this section we fix $A\in \Lambda (\mu, \lambda,\tau)$
and let $\mathcal{L}=-\text{div}(A\nabla)$,
$\Gamma(X,Y)=\Gamma_A (X,Y)$.
Let $\Omega$ be a bounded Lipschitz domain in $\mathbb{R}^{d}$
and $n=(n_1, \cdots, n_d)$ the outward unit normal to $\partial\Omega$.
For $f\in L^p (\partial\Omega, \mathbb{R}^m)$, the single layer potential
$\mathcal{S}(f)=\mathcal{S}_A (f)=(u^1, \dots, u^m)$ is defined by
\begin{equation}\label{single-layer-potential}
u^\alpha(X)
=\int_{\partial\Omega}
\Gamma^{\alpha\beta} (X,Y) f^\beta(Y)\, d\sigma (Y),
\end{equation}
while the double layer potential $\mathcal{D} (f)
=\mathcal{D}_A (f)=(w^1, \dots, w^m)$ is defined by
\begin{equation}\label{doubel-layer-potential}
\aligned
w^\alpha (X)
& =\int_{\partial\Omega}
n_j (Y) a_{ij}^{\beta\gamma}(Y)
\frac{\partial}{\partial y_i}
\Gamma^{\alpha\beta} (X,Y) f^\gamma (Y)\, d\sigma (Y)\\
& =\int_{\partial\Omega}
\left( \frac{\partial}{\partial \nu_{A^*}}
\Gamma_{A^*}^\alpha (Y,X)\right)^\gamma f^\gamma (Y) \,
d\sigma (Y),
\endaligned
\end{equation}
where $\Gamma_{A^*}^\alpha =(\Gamma_{A^*}^{1,\alpha}, \dots,
\Gamma_{A^*}^{m,\alpha})$ and 
$\Gamma_{A^*} (X,Y)$ is the fundamental solution
for $\mathcal{L}^*=-\text{div}(A^*\nabla)$, with pole at $Y$.
Clearly, both $\mathcal{S} (f)$ and $\mathcal{D} (f)$
are solutions of $\mathcal{L} (u)=0$ in $\mathbb{R}^{d}\setminus \partial \Omega$.

The definitions of single and layer potentials are motivated by the following
Green's representation formula.

\begin{prop}\label{Cauchy-representation-lemma}
Let $u\in C^1 (\overline{\Omega})$. 
Suppose that $\mathcal{L}(u)=0$ in $\Omega$.
Then for any $X\in \Omega$,
\begin{equation}\label{Cauchy-representation-formula}
u(X)=\mathcal{S}\big(\frac{\partial u}{\partial \nu}\big) (X) -\mathcal{D}(u)(X),
\end{equation}
where $\frac{\partial u}{\partial \nu}$ denotes the conormal derivative 
of $u$ on $\partial\Omega$, defined by $\left(\frac{\partial u}{\partial\nu}\right)^\alpha
=n_i a_{ij}^{\alpha\beta} \frac{\partial u^\beta}{\partial x_j}$.
\end{prop}

\begin{proof}
Fix $X\in \Omega$. Choose $r>0$ so small that $B(X,4r)\subset \Omega$.
Let $\varphi\in C_0^\infty (B(X,2r))$ be such that $\varphi=1$ on $B(X,r)$.
It follows from (\ref{fundamental-solution-representation}) that
\begin{equation}
\aligned
u^\gamma (X)
= & (u\varphi)^\gamma (X)
=\int_{\Omega}
a_{ji}^{\beta\alpha}(Y) \frac{\partial}{\partial y_j}
\Gamma_{A^*}^{\beta\gamma} (Y,X) \cdot \frac{\partial }{\partial y_i} 
(u^\alpha \varphi )\, dY\\
& =\int_\Omega a_{ji}^{\beta\alpha} (Y) \frac{\partial}{\partial y_j}
\Gamma_{A^*}^{\beta\gamma} (Y,X)\cdot \frac{\partial u^\alpha}{\partial y_i}\, dX\\
&\qquad\qquad +\int_\Omega a_{ji}^{\beta\alpha} (Y) \frac{\partial}{\partial y_j}
\Gamma_{A^*}^{\beta\gamma} (Y,X)\cdot \frac{\partial }{\partial y_i}
\big\{ u^\alpha (\varphi-1)\big\} \, dX.\\
\endaligned
\end{equation}
Using integration by parts and $\mathcal{L}(u)=0$ in $\Omega$, we obtain
\begin{equation}\label{4-3}
\aligned
u^\gamma (X)
&=\int_{\partial\Omega}
n_j(Y) a_{ji}^{\beta\alpha}(Y)
\frac{\partial u^\alpha}{\partial y_i}
\cdot \Gamma_{A^*}^{\beta\gamma} (Y,X)\, d\sigma (Y)\\
&\qquad -\int_{\partial\Omega}
n_i(Y) a_{ji}^{\beta\alpha}(Y)
\frac{\partial}{\partial y_j}
\Gamma_{A^*} ^{\beta\gamma} (Y,X)\cdot u^\alpha (Y) \, d\sigma (Y).
\endaligned
\end{equation}
Since $\Gamma^{\alpha\beta}_{A^*} (Y,X)=\Gamma_A^{\beta\alpha}(X,Y)$,
this gives (\ref{Cauchy-representation-formula}).
\end{proof}

\begin{remark}\label{Green-representation-remark}
{\rm
Suppose that $u\in C^1(\overline{\Omega})$ 
and $\mathcal{L}(u)=-\frac{\partial f_i}{\partial x_i} + g$, where
$f_i, g\in C(\overline{\Omega}, \mathbb{R}^m)$ and $f_i=0$ on $\partial\Omega$.
Then
\begin{equation}\label{Green-representation-formula-1}
 u(X)=\mathcal{S}\big(\frac{\partial u}{\partial \nu}\big) (X) -\mathcal{D}(u)(X) +v(X),
\end{equation}
where
$$
v^\alpha (X)=\int_\Omega  \frac{\partial}{\partial y_i}
\Gamma ^{\alpha\beta} (X,Y) \cdot f_i^\beta (Y)\, dY
+\int_\Omega \Gamma^{\alpha\beta} (X,Y) g^\beta (Y)\, dY.
$$
}
\end{remark}

\begin{thm}\label{layer-potential-maximal-theorem}
Let $1<p<\infty$. Then
\begin{equation}\label{layer-potential-maximal-estimate}
\| \big(\nabla \mathcal{S}(f)\big)^*\|_p
+\| \big( \mathcal{D}(f)\big)^*\|_p \le
C_p \| f\|_p,
\end{equation}
where $C_p$ depends only on $d$, $m$, $\mu$, $\lambda$, $\tau$, $p$
and the Lipschitz character of $\Omega$.
\end{thm}

\begin{proof} This follows readily from Theorem \ref{maximal-function-theorem}.
\end{proof}

For a function $u$ defined in $\mathbb{R}^{d}\setminus \partial\Omega$,
we will use $u_+$ and $u_-$ to denote its nontangential limits on $\partial\Omega$,
taken inside $\Omega$ and outside $\overline{\Omega}$ respectively.

\begin{thm}\label{trace-theorem}
Let $u=\mathcal{S}(f)$ for some $f\in L^p(\partial\Omega)$ and $1<p<\infty$.
Then for a.e. $P\in \partial\Omega$,
\begin{equation}\label{trace-formula}
\left(\frac{\partial u^\alpha}{\partial x_i}\right)_\pm  (P)
= \pm \frac12 n_i (P) b^{\alpha\beta}(P) f^\beta (P)
+\text{\rm p.v.}\int_{\partial\Omega}
\frac{\partial}{\partial P_i} \Gamma^{\alpha\beta } (P, Y) f^\beta (Y)\, d\sigma (Y),
\end{equation}
where $(b^{\alpha\beta} (P))_{m\times m}$
is the inverse matrix of $\big(a^{\alpha\beta}_{ij}(P)n_i(P)n_j(P)\big)_{m\times m}$.
\end{thm}

\begin{proof}
By Theorem \ref{layer-potential-maximal-theorem} we may assume that $f$ is a Lipschitz
function on $\partial\Omega$. 
Also it is known that
there exists  a set $F\subset \partial\Omega$ such that
$\sigma (\partial \Omega -F)=0$ and the trace formula (\ref{trace-formula})
holds for any $P\in F$ and for any $\Gamma (X,Y)=
\Gamma_E (X,Y)$ with constant matrix $E\in \Lambda(\mu, \lambda,\tau)$
(see e.g. \cite{Gao-1991, Dahlberg-Verchota-1990, Mitrea-Taylor-1999}).

Now fix $P\in F$ and
\begin{equation}
v^\alpha (X)= \int_{\partial\Omega}
\Theta^{\alpha\beta}(X,Y; A(P)) f^\beta (Y)\, d\sigma(Y)
\end{equation}
be the single layer potential for the elliptic operator $\mathcal{L}^{A(P)}
=-\text{div}(A(P)\nabla)$ with constant coefficients.
In view of (\ref{local-behavior-estimate-1}) 
and (\ref{constant-coefficient-difference-estimate}), we have
\begin{equation}
\aligned
& |\nabla_1 \Gamma_A (X,Y)-\nabla_1 \Theta(X,Y; A(P))|\\
&  \le |\nabla_1 \Gamma_A (X,Y)-\nabla_1\Theta (X,Y, A(Y))|
+|\nabla_1\Theta (X,Y; A(Y))-\nabla_1\Theta(X,Y; A(P))|\\
& 
\le C|X-Y|^{1-d+\lambda} +C |X-Y|^{1-d} |Y-P|^\lambda \\
& \le C|P-Y|^{1-d+\lambda}
\endaligned
\end{equation}
for any $X\in \gamma (P)=\{ Z\in \mathbb{R}^d\setminus \partial\Omega:
\text{ dist} (Z, \partial \Omega)<C_0 |Z-P|\}$ and $Y\in \partial\Omega$.
By Lebesgue's dominated convergence theorem, this implies that
\begin{equation}\label{4-5}
\aligned
(\nabla u^\alpha)_\pm (P)
=  &(\nabla v^\alpha)_\pm (P)
+ \int_{\partial\Omega}
\big\{
\nabla_1 \Gamma_A^{\alpha\beta} (P,Y)-\nabla_1 \Theta^{\alpha\beta} (P,Y; A(P))\big\}
f^\beta (Y)\, d\sigma (Y)\\
= & \pm \frac12 n (P) b^{\alpha\beta}(P) f^\beta P)
+\text{\rm p.v.}\int_{\partial\Omega}
\nabla_1\Gamma^{\alpha\beta }_A (P, Y) f^\beta (Y)\, d\sigma (Y).
\endaligned
\end{equation}
This finishes the proof.
\end{proof}

It follows from (\ref{trace-formula}) that if $u=\mathcal{S}(f)$,
\begin{equation}\label{tangential-relation}
n_j \left(\frac{\partial u^\alpha}{\partial x_i}\right)_+
-
n_i \left(\frac{\partial u^\alpha}{\partial x_j}\right)_+
=
n_j \left(\frac{\partial u^\alpha}{\partial x_i}\right)_-
-
n_i \left(\frac{\partial u^\alpha}{\partial x_j}\right)_-;
\end{equation}
i.e. $(\nabla_{tan} u)_+ = (\nabla_{tan} u)_-$ on $ \partial\Omega$.
Moreover, 
let $\left( \frac{\partial u}{\partial\nu}\right)^\alpha_\pm 
=n_i a_{ij}^{\alpha\beta} \left(\frac{\partial u^\beta}{\partial x_j}\right)_\pm $
on $\partial\Omega$.
Then $\left ( \frac{\partial u}{\partial\nu} \right)_\pm 
=(\pm \frac12 I +\mathcal{K}_A) (f)$, where
\begin{equation}\label{operator-K}
\big( \mathcal{K}_A (f)(P)\big)^\alpha
=\text{\rm p.v.}
\int_{\partial\Omega} K_A^{\alpha\beta} (P,Y) f^\beta (Y)\, d\sigma (Y)
\end{equation}
and
\begin{equation}\label{kernel-K}
K_A ^{\alpha\beta}
(P,Y)
=n_i(P) a_{ij}^{\alpha\gamma} (P)
\frac{\partial}{\partial P_j}
\Gamma_A^{\gamma\beta} (P,Y).
\end{equation}
In particular we have the jump relation
\begin{equation}\label{jump-relation}
f=\left(\frac{\partial u}{\partial\nu}\right)_+-\left(\frac{\partial u}{\partial\nu}\right)_-.
\end{equation}
We may deduce from Theorem \ref{operator-boundedness-theorem} that $\| \mathcal{K}_A (f)\|_p
\le C_p \| f\|_p$, where
$C_p$ depends only on $d$, $m$, $\mu$, $\lambda$, $\tau$, $p$ and
the Lipschitz character of $\Omega$.

\begin{remark}\label{mean-value-remark}
{\rm
Let $u=\mathcal{S}_A(f)$ for some $f\in L^p(\partial\Omega, \mathbb{R}^m)$ and $1<p<\infty$.
Then $\int_{\partial\Omega}\left(\frac{\partial u}{\partial\nu}\right)_+\, d\sigma =0$.
It follows that $((1/2)I+\mathcal{K}_A)\big( L^p(\partial\Omega, \mathbb{R}^m)\big)
\subset L^p_0(\partial\Omega, \mathbb{R}^m)$.
This implies that 
$$
\mathcal{K}_A \big(L^p_0(\partial\Omega, \mathbb{R}^m)\big)
\subset L^p_0(\partial\Omega, \mathbb{R}^m) \qquad \text{ for } 1<p<\infty.
$$
}
\end{remark}

Let $W^{1,p}(\partial\Omega,\mathbb{R}^m)$ denote the subspace of functions $f$ in
$L^p(\partial\Omega, \mathbb{R}^m)$ with tangential derivatives $\nabla_{tan} f$ in $L^p(\partial\Omega)$,
equipped with the scale-invariant norm
\begin{equation}\label{W-norm}
\| f\| _{1,p}
=\| \nabla_{tan} f\|_p + \big[ \sigma(\partial\Omega)\big]^{\frac{1}{1-d}} \| f\|_p.
\end{equation}
It follows from Theorem \ref{operator-boundedness-theorem} that for $1<p<\infty$,
\begin{equation}\label{W-estimate}
\| \mathcal{S}(f)\|_{1,p} \le C\| f\|_p
\end{equation}
where $C$ depends only on $d$, $m$,
$\mu$, $\lambda$, $\tau$, $p$ and the Lipschitz character of $\Omega$.

The next theorem gives the trace of the double layer potentials.

\begin{thm}\label{double-layer-trace-theorem}
Let $w=\mathcal{D}(f)$ where $f\in L^p(\partial\Omega)$ and $1<p<\infty$.
Then $$
w_\pm =\big( \mp \frac12 I +\mathcal{K}_{A^*}^*\big) (f) \qquad
\text{ on } \partial\Omega,
$$ 
where
$\mathcal{K}^*_{A^*}$ is the adjoint operator of $\mathcal{K}_{A^*}$, defined
by (\ref{operator-K}) and (\ref{kernel-K}).
\end{thm}

\begin{proof}
Note that if $E$ is a constant matrix in $\Lambda(\mu, \lambda,\tau)$, then
\begin{equation}
\frac{\partial }{\partial x_i} \Theta(X,Y;E)
=-\frac{\partial}{\partial y_i} \Theta (X,Y;E).
\end{equation}
This, together with (\ref{local-behavior-estimate-1})
and (\ref{local-behavior-estimate-2}), shows that
\begin{equation}
\big|\frac{\partial}{\partial y_i}
\Gamma_A^{\alpha\beta} (X,Y)
+
\frac{\partial}{\partial x_i}
\Gamma_A^{\alpha\beta} (X,Y)\big|
\le C|X-Y|^{1-d+\lambda},
\end{equation}
for any $X,Y\in \mathbb{R}^d$.
Thus, as in the proof of Theorem \ref{trace-theorem}, it follows from the
Lebesgue's dominated convergence theorem that
$$
w^\alpha_\pm (P)
=-v^\alpha_\pm (P)
+
\int_{\partial\Omega}
n_j(Y) a_{ij}^{\beta\gamma}(Y)
\left\{ \frac{\partial}{\partial y_i}
\Gamma^{\alpha\beta}(P,Y)
+\frac{\partial}{\partial P_i}
\Gamma^{\alpha\beta} (P,Y)\right\} f^\gamma (Y)\, d\sigma (Y),
$$
where
\begin{equation}
v^\alpha (X)
=\frac{\partial}{\partial x_i}
\int_{\partial\Omega}
n_j (Y) a_{ij}^{\beta\gamma} (Y)
\Gamma^{\alpha\beta} (X,Y) f^\gamma (Y)\, d\sigma (Y).
\end{equation}
In view of the trace formula (\ref{trace-formula}), we have
\begin{equation}
\aligned
v_\pm^\alpha (P)
= & \pm \frac12 n_i(P) b^{\alpha\beta} (P) \cdot n_j(P) a_{ij}^{\beta\gamma}(P)
f^\gamma (P)\\
& \qquad
+\text{\rm p.v.}
\int_{\partial\Omega}
n_j (Y) a_{ij}^{\beta\gamma} (Y)
\frac{\partial}{\partial P_i}
\Gamma^{\alpha\beta} (P,Y)\cdot f^\gamma (Y)\, d\sigma (Y)\\
=&
\pm \frac12 f^\alpha (P)
+\text{\rm p.v.}
\int_{\partial\Omega}
n_j (Y) a_{ij}^{\beta\gamma} (Y)
\frac{\partial}{\partial P_i}
\Gamma^{\alpha\beta} (P,Y)\cdot f^\gamma (Y)\, d\sigma (Y).
\endaligned
\end{equation}
Thus
\begin{equation}
\aligned
w^\alpha_\pm (P)
 & =\mp\frac12
f^\alpha (P)
+\text{\rm p.v.}
\int_{\partial\Omega}
n_j (Y) a_{ij}^{\beta\gamma} (Y)
\frac{\partial}{\partial y_i}
\Gamma^{\alpha\beta} (P,Y)\cdot f^\gamma (Y)\, d\sigma (Y)\\
&
=\mp \frac12 f^\alpha (P)
+\text{\rm p.v.}
\int_{\partial\Omega}
K_{A^*}^{\beta\alpha} (Y,P) f^\beta (Y)\, d\sigma (Y),
\endaligned
\end{equation}
where $K_{A^*}^{\beta\alpha} (Y,P)$ is defined by
(\ref{kernel-K}), but with $A$ replaced by $A^*$.
This completes the proof.
\end{proof}

In summary, if $1<p<\infty$ and $f\in L^p(\partial\Omega)$,
then $u=\mathcal{S}(f)$ is a solution to the $L^p$ Neumann problem
 in $\Omega$ with boundary data $((1/2) I +\mathcal{K}_A)f$, while
$w=\mathcal{D}(f)$ is a solution to the $L^p$ Dirichlet problem
in $\Omega$ with boundary data $(-(1/2) I +\mathcal{K}_{A^*}^*)f$.
Furthermore, $(1/2)I +\mathcal{K}_A: L^p_0(\partial\Omega, \mathbb{R}^m)
\to L^p_0(\partial\Omega, \mathbb{R}^m)$
and $-(1/2)I +\mathcal{K}^*_{A^*}: L^p(\partial\Omega, \mathbb{R}^m)
\to L^p(\partial\Omega, \mathbb{R}^m)$ are bounded.
As a result, one may establish the existence of solutions in the
$L^p$ Neumann and Dirichlet problems in $\Omega$ by showing that
the operators $(1/2)I +\mathcal{K}_A$ and $-(1/2)I +\mathcal{K}_{A^*}^*$
are invertible on $L^p_0(\partial\Omega, \mathbb{R}^m)$ and 
$L^p(\partial\Omega, \mathbb{R}^m)$ respectively.
This is the so-called method of layer potentials.

In the remaining of this section we discuss the layer potentials for
 $\mathcal{L}_\varep=-\text{div}(A(\varep^{-1}X)\nabla)$.
Let $\Gamma_\varepsilon (X,Y)=\Gamma_{\varep, A} (X,Y)$ 
denote the matrix of fundamental solutions for
the operator $\mathcal{L}_\varepsilon$ on $\mathbb{R}^{d}$,
with pole at $Y$. By rescaling we have
\begin{equation}\label{e-relation}
\Gamma_\varepsilon (X,Y)=\varepsilon^{2-d} \Gamma(\varepsilon^{-1}X, 
\varepsilon^{-1}Y).
\end{equation}
Thus, by (\ref{size-estimate}),
\begin{equation}\label{e-size-estimate}
\aligned
& |\Gamma_\varepsilon (X,Y)|\le C |X-Y|^{2-d},\\
|\nabla_X \Gamma_\varepsilon (X,Y)|
& +
|\nabla_Y \Gamma_\varepsilon (X,Y)|
\le C |X-Y|^{1-d}
\endaligned
\end{equation}
for any $X,Y\in \mathbb{R}^{d}$.

For $f\in L^p (\partial\Omega)$, the single layer potential
$\mathcal{S}_\varepsilon(f)=(\mathcal{S}_\varepsilon^1(f),
 \dots, \mathcal{S}_\varepsilon^m( f))$ is defined by
\begin{equation}\label{e-single-layer-potential}
\mathcal{S}^\alpha_\varepsilon (f) (X)
=\int_{\partial\Omega}
\Gamma_\varepsilon^{\alpha\beta} (X,Y) f^\beta (Y)\, d\sigma (Y), \quad
\end{equation}
while the double layer potential $\mathcal{D}_\varep (f)
=(\mathcal{D}^1_\varep (f), \dots, \mathcal{D}_\varep^m (f))$ is defined by
\begin{equation}\label{e-doubel-layer-potential}
\mathcal{D}^\alpha_\varep (f) (X)
=\int_{\partial\Omega}
n_j(Y) a_{ij}^{\beta\gamma}(\varepsilon^{-1}Y)
\frac{\partial}{\partial y_i}
\Gamma_\varep^{\alpha\beta} (X,Y) f^\gamma (Y)\, d\sigma (Y).
\end{equation}
Clearly, both $\mathcal{S}_\varep (f)$ and $\mathcal{D}_\varep (f)$
are solutions of $\mathcal{L}_\varep (u)=0$ in $\mathbb{R}^{d}\setminus \partial \Omega$.

\begin{thm}\label{e-layer-potential-maximal-theorem}
Let $1<p<\infty$. Then
$$
\| \big(\nabla \mathcal{S}_\varepsilon (f)\big)^*\|_p +\| \big(\mathcal{D}_\varep (f)\big)^*\|_p
\le C_p\| f\|_p,
$$
where $C_p$ depends only on $d$, $m$, $\mu$, $\lambda$, $\tau$, $p$ and the Lipschitz
character of $\Omega$.
\end{thm}

\begin{proof} Fix $\varep>0$ and define
\begin{equation}\label{e-domain}
\Omega_\varep
=\big\{ \varep^{-1}X:\ X\in \Omega\big\}.
\end{equation}
Let $u_\varep=\mathcal{S}_\varep (f)$.
It follows from (\ref{e-relation}) that
$u_\varep (X)=\varep v(\varep^{-1}X)$, where
$v(x)$ is the single layer potential on $\partial\Omega_\varep$ for the operator $\mathcal{L}$ 
with density $g$ given by $g(Y)=f(\varep Y)$.
Since $\Omega_\varep$ and $\Omega$ share the same Lipschitz character,
$\| (\nabla v)^*\|_{L^p(\partial\Omega_\varep)} \le C \| g\|_{L^p(\partial\Omega_\varep)}$,
where $C$ depends only on on $d$, $m$, $\mu$, $\lambda$, $\tau$, $p$ and the Lipschitz
character of $\Omega$, not on $\varep$.
By rescaling, this gives the desired estimate for $\nabla \mathcal{S}_\varep (f)$.
The estimate on $\mathcal{D}_\varep (f)$ follows in the same manner.
\end{proof}

The next theorem follows readily from Theorems \ref{trace-theorem}
and \ref{double-layer-trace-theorem} by rescaling.

\begin{thm}\label{e-trace-theorem}
Let $1<p<\infty$ and $f\in L^p(\partial\Omega)$.
Let $u_\varep =\mathcal{S}_\varep (f)$. Then $(\nabla u_\varep)_\pm (P)$
exists for a.e. $P\in \partial\Omega$ and
$\left(\frac{\partial u_\varep}{\partial \nu_\varep}\right)_\pm 
=(\pm \frac12 I +\mathcal{K}_{\varep, A}) (f)$, where 
\begin{equation}\label{e-single-layer-potential-trace}
\mathcal{K}_{\varep,A}^\alpha (f) (P)
=\text{\rm p.v.}
\int_{\partial\Omega_\varep}
K_A^{\alpha\beta} (\varep^{-1} P,  Y) f^\beta (\varep Y)\, d\sigma (Y)
\end{equation}
and the integral kernel $K_A^{\alpha\beta} (P,Y)$ on $\partial\Omega_\varep
\times \partial\Omega_\varep$ is given by
(\ref{kernel-K}). Similarly, if
$w_\varep =\mathcal{D}_\varep (f)$, then
$(w_\varep)_\pm =(\mp\frac12 I +\mathcal{K}_{\varep, A^*}^*)f$.
\end{thm}

\section{Rellich property}

In this section we reduce the solvability of the $L^2$
Neumann, Dirichlet and regularity problems
for $\mathcal{L}(u)=0$ in Lipschitz domains
to certain boundary Rellich estimates.

\begin{definition}
{\rm
Let $\mathcal{L}=-\text{div}(A(X)\nabla)$
and $\Omega$ be a bounded Lipschitz domain
with connected boundary.
We say that $\mathcal{L}$ has the Rellich property in $\Omega$
with constant $C=C(\Omega)$ if 
$ \|\nabla u \|_2 \le C \|\frac{\partial u}{\partial\nu}\|_2$
and $\|\nabla u\|_2 \le C\|\nabla_{tan} u\|_2$,
whenever $u$ is a solution to $\mathcal{L}(u)=0$ in $\Omega$
such that $(\nabla u)^*\in L^2(\partial\Omega)$ and
$\nabla u$ exists n.t. on $\partial\Omega$.
}
\end{definition}

\begin{definition}
{\rm
We say that the Dirichlet problem $(D)_p$ for $\mathcal{L}(u)=0$ in $\Omega$ is uniquely
solvable with estimate $\|(u)^*\|_p \le C\| u\|_p$, if for any $f\in L^p(\partial\Omega,
\mathbb{R}^m)$, there exists a unique solution to $\mathcal{L}(u)=0$
in $\Omega$ with the property that $(u)^*\in L^p(\partial\Omega)$
and $u=f$ n.t. on $\partial\Omega$, and the solution satisfies
$\|(u)^*\|_p \le C\| f\|_p$.

We say that the regularity problem $(R)_p$ for $\mathcal{L}(u)=0$ in $\Omega$ is uniquely
solvable with estimate $\|(\nabla u)^*\|_p \le C\| u\|_{1,p}$, if for any $f\in W^{1,p}(\partial\Omega,
\mathbb{R}^m)$, there exists a unique solution to $\mathcal{L}(u)=0$
in $\Omega$ with the property that $(\nabla u)^*\in L^p(\partial\Omega)$
and $u=f$ n.t. on $\partial\Omega$, and the solution satisfies
$\|(u)^*\|_p \le C\| f\|_{1,p}$.

We say that the Neumann problem $(N)_p$ for $\mathcal{L}(u)=0$ in $\Omega$ is uniquely
solvable with estimate $\|(\nabla u)^*\|_p \le C\| \frac{\partial u}{\partial\nu}\|_p$, 
if for any $f\in L_0^p(\partial\Omega,
\mathbb{R}^m)$, there exists a solution, unique up to constants, to $\mathcal{L}(u)=0$
in $\Omega$ with the property that $(\nabla u)^*\in L^p(\partial\Omega)$
and $\frac{\partial u}{\partial\nu}=f$ n.t. on $\partial\Omega$, and the solution satisfies
$\|(\nabla u)^*\|_p \le C\| f\|_p$.
}
\end{definition}

The following two theorems are the main results of this section.
The first theorem treats the solvability in small scale - the constant $C$
in the nontangential-maximal-function estimates 
in (\ref{Neumann-regularity-estimate}) depends on diam$(\Omega)$, if
diam$(\Omega)\ge 1$. 
The estimates in the second theorem
are scale-invariant.
As a result, by rescaling,
they leads to uniform estimates 
in a Lipschitz domain
for the family of elliptic operators $\{ \mathcal{L}_\varep\}$.

\begin{thm}\label{local-Rellich-imply-solvability}
Let $\mathcal{L}=-\text{div}(A\nabla)$ with $A
\in \Lambda (\mu, \lambda, \tau)$ and $A^*=A$. Let $R\ge 1$.
Suppose that for any Lipschitz domain $\Omega$ with diam$(\Omega)\le (1/4)$
and connected boundary, there exists
$C(\Omega)$ depending only on the Lipschitz character of $\Omega$ such that for each $s\in (0,1]$,
$\mathcal{L}_s =-\text{div}\big( (sA+(1-s)I)\nabla \big)$
 has the Rellich property in $\Omega$ with 
constant $C(\Omega)$.
Then for any Lipschitz domain $\Omega$ with diam$(\Omega)\le R$ and
connected boundary, $(R)_2$ and $(N)_2$ 
for $\mathcal{L}(u)=0$ in $\Omega$ are uniquely solvable and
the solutions satisfy the estimates
\begin{equation}\label{Neumann-regularity-estimate}
\| (\nabla u)^*\|_2 \le C \| \frac{\partial u}{\partial \nu}\|_2
\quad \text{ and }
\quad
\|(\nabla u)^*\|_2 \le C \| \nabla_{tan} u\|_2,
\end{equation}
where $C$ depends only on $\mu$, $\lambda$, $\tau$, the 
Lipschitz character of $\Omega$ and $R$ (if $\text{diam}(\Omega)\ge 1$).
Furthermore, the $L^2$ Dirichlet problem in $\Omega$ is uniquely solvable
with the estimate $\|(u)^*\|_2 \le C \| u\|_2$.
\end{thm}

Recall that $C$ is called a ``good'' constant if it depends only on
$d$, $m$, $\mu$, $\lambda$, $\tau$ and the Lipschitz character of $\Omega$.

\begin{thm}\label{Rellich-imply-solvability}
Let $\mathcal{L}=-\text{div}(A\nabla)$ with $A
\in \Lambda (\mu, \lambda, \tau)$ and $A^*=A$.
Suppose that for any Lipschitz domain $\Omega$ with connected boundary, there exists a ``good''
constant
$C(\Omega)$ such that for each $s\in (0,1]$,
$\mathcal{L}_s =-\text{div}\big( (sA+(1-s)I)\nabla \big)$
 has the Rellich property in $\Omega$ with 
constant $C(\Omega)$.
Then for any Lipschitz domain $\Omega$ with connected boundary,
$(R)_2$ and $(N)_2$ 
for $\mathcal{L}(u)=0$ in $\Omega$ are uniquely solvable and
the solutions satisfy the estimates in (\ref{Neumann-regularity-estimate})
with a ``good'' constant $C$.
Furthermore, the $L^2$ Dirichlet problem in $\Omega$ is uniquely solvable
with the estimate $\|(u)^*\|_2 \le C \| u\|_2$ for a ``good''
constant $C$.
\end{thm}

The uniqueness for $(R)_2$ and $(N)_2$ follows readily
from the Green's identity,
\begin{equation}\label{Green-identity}
\int_\Omega a_{ij}^{\alpha\beta} \frac{\partial u^\alpha}{\partial x_i}
\cdot \frac{\partial u^\beta}{\partial x_j}\, dX
=\int_{\partial \Omega}
\left(\frac{\partial u}{\partial\nu}\right)^\alpha
 u^\alpha \, d\sigma,
\end{equation}
by approximating $\Omega$ from inside.
We will use the method of layer potentials
to establish the existence of solutions in Theorems
\ref{local-Rellich-imply-solvability} and \ref{Rellich-imply-solvability}. 

\begin{lemma}\label{local-interior-exterior-lemma} {\rm (Rellich estimates for small scales)}
Let $\Omega$ be a bounded Lipschitz domain with $r_0=\text{diam}(\Omega)\le R$.
Suppose that $\mathcal{L}(u)=0$ in $\Omega_\pm$, $(\nabla u)^*\in L^2(\partial\Omega)$
and $(\nabla u)_\pm $ exists n.t. on $\partial\Omega$.
Under the same conditions on $A$ as in Theorem \ref{local-Rellich-imply-solvability},
we have
\begin{equation}\label{interior-exterior-Rellich}
\aligned
& \int_{\partial \Omega}
 |(\nabla u)_\pm|^2\, d\sigma
 \le C \int_{\partial\Omega}
|\left(\frac{\partial u}{\partial \nu}\right)_\pm|^2\, d\sigma
+\frac{C}{r_0}
\int_{N_\pm} |\nabla u|^2\, dX,\\
& \int_{\partial \Omega}
|(\nabla u)_\pm|^2\, d\sigma
\le  C \int_{\partial\Omega}
|(\nabla_{tan} u)_\pm|^2\, d\sigma
+\frac{C}{r_0}
\int_{N_\pm} |\nabla u|^2\, dX,
\endaligned
\end{equation}
where $N_\pm =\{ X\in \Omega_\pm: \text{dist}(X, \partial\Omega)\le r_0\}$, and
$C$ depends only on $d$, $m$, $\mu$, $\lambda$, $\tau$,
the Lipschitz character of $\Omega$, and $R$ (if $r_0> 1$).
\end{lemma}

\begin{proof}
Let $\psi: \mathbb{R}^{d-1}\to \mathbb{R}$ be a Lipschitz function such that
$\psi(0)=0$ and $\|\nabla \psi\|_\infty\le M$.
Let
\begin{equation}\label{definition-of-Z}
\aligned
& Z(r)=\big\{ (x^\prime, x_d)\in \mathbb{R}^d:
\ |x^\prime|<r \text{ and } \psi(x^\prime)<x_d <10 \sqrt{d} (M+1) r\big\},\\
& \Delta(r) =\big\{ (x^\prime, \psi(x^\prime))\in \mathbb{R}^d: \ |x^\prime|<r\big\}.
\endaligned
\end{equation}
Suppose that $\mathcal{L}(u)=0$ in $\Omega_0=
Z(3r)$, $(\nabla u)^*\in L^2(\partial\Omega_0)$ and $\nabla u$
exists n.t. on $\partial\Omega_0$.
Assume that $\text{diam}(Z(2r))<(1/4)$. Then for any $t\in (1,2)$,
$\mathcal{L}$ has the Rellich property in the Lipschitz domain $Z(tr)$
with constant $C_0= C(Z(tr))$ depending only on $M$.
It follows that
\begin{equation}\label{R-1}
\aligned
\int_{\Delta(r)} |\nabla u|^2\, d\sigma
& \le \int_{\partial Z(tr)} |\nabla u|^2\, d\sigma\\
&\le  C_0 \int_{\Delta (2r)} |\frac{\partial u}{\partial\nu}|^2\, d\sigma
+C C_0\int_{\partial Z(tr) \setminus \Delta (tr)} |\nabla u|^2\, d\sigma.
\endaligned
\end{equation}
We now integrate both sides of (\ref{R-1}) with respect to $t$ over the interval $(1,2)$
to obtain
\begin{equation}\label{R-2}
\int_{\Delta(r)}
|\nabla u|^2\, d\sigma 
\le C \int_{\Delta(2r)}
|\frac{\partial u}{\partial\nu}|^2\, d\sigma
+\frac{C}{r}\int_{Z(2r)}
|\nabla u|^2\, dX.
\end{equation}

Finally we choose $r=c(M)r_0$ if $r_0\le 1$, and $r=c(M)$ if
$ r_0>1$. The first inequality in (\ref{interior-exterior-Rellich}) follows
from (\ref{R-2}) by covering $\partial\Omega$ with $\{\Delta_i\}$, each of which 
may be obtained
from $\Delta(r)$ by translation and rotation.
The proof for the second inequality in (\ref{interior-exterior-Rellich})
is similar.
\end{proof}

\begin{remark}\label{interior-exterior-Rellich-remark}
{\rm
Under the same conditions on $A$ as in Theorem \ref{Rellich-imply-solvability},
the estimates in (\ref{interior-exterior-Rellich})
hold with constant $C$ independent of $R$.
This is because we may choose $r=c(M)r_0$ for any $\Omega$.
}
\end{remark}

\begin{lemma}\label{invertibility-of-Neumann-operator}
Let $R\ge 1$ and
$\Omega$ be a bounded Lipschitz domain with diam$(\Omega)\le R$.
Under the same conditions on $A$ as in Theorem \ref{local-Rellich-imply-solvability},
the operators $(1/2)I +\mathcal{K}_{A}: L^2_0(\partial\Omega, \mathbb{R}^m)
\to L^2_0(\partial\Omega, \mathbb{R}^m)$ and
$-(1/2)I +\mathcal{K}_A: L^2(\partial\Omega, \mathbb{R}^m)\to
L^2(\partial\Omega, \mathbb{R}^m)$ are invertible and
\begin{equation}\label{inverse-bound}
\aligned
& \| \big((1/2)I +\mathcal{K}_A\big)^{-1}\|_{L^2_0\to L^2_0} \le C,\\
 &\| \big(-(1/2)I +\mathcal{K}_A\big)^{-1}\|_{L^2\to L^2} \le C,\\
\endaligned
\end{equation}
where $C$ depends only on $\mu$, $\lambda$, $\tau$,
the Lipschitz character of $\Omega$, and $R$ (if $\text{diam}(\Omega)\ge 1$).
\end{lemma}

\begin{proof}
Let $f\in L^2_0(\partial\Omega,\mathbb{R}^m)$ and $u=\mathcal{S}_A (f)$.
Then $\mathcal{L} (u)=0$ in $\mathbb{R}^d\setminus \partial\Omega$,
$(\nabla u)^*\in L^2(\partial\Omega)$ and
$(\nabla u)_\pm $ exists n.t. on $\partial\Omega$.
Also recall that $(\nabla_{tan} u)_+=(\nabla_{tan} u)_-$ on $\partial\Omega$.
Since $|u(X)| + |X||\nabla u(X)|
=O(|X|^{2-d})$ as $|X|\to\infty$, it follows from integration
by parts that
\begin{equation}\label{exterior-Green-identity}
\int_{\Omega_-} a_{ij}^{\alpha\beta} \frac{\partial u^\alpha}{\partial x_i}
\cdot \frac{\partial u^\beta}{\partial x_j}\, dX
=-\int_{\partial \Omega}
\left(\frac{\partial u}{\partial\nu}\right)_-^\alpha
 u^\alpha_- \, d\sigma.
\end{equation}
By the jump relation (\ref{jump-relation}), 
$\int_{\partial\Omega} \left(\frac{\partial u}
{\partial\nu}\right)_-d\sigma=-\int_{\partial\Omega} fd\sigma =0$.
Using Poincar\'e's inequality on $\partial\Omega$,
(\ref{exterior-Green-identity}), together with (\ref{Green-identity}), gives
\begin{equation}\label{Green-estimate}
\int_{\Omega_\pm}
|\nabla u|^2\, dX
\le Cr_0 \|\left(\frac{\partial u}{\partial \nu}\right)_\pm\|_2 \|\nabla_{tan} u\|_2.
\end{equation}
By combining (\ref{interior-exterior-Rellich}) with (\ref{Green-estimate}) and
then using the Cauchy inequality with an $\varep>0$,
we see that
$\| (\nabla u)_\pm \|_2 \le C\|\left(\frac{\partial u}{\partial \nu}\right)_\pm\|_2$
and $\|(\nabla u)_\pm \|_2 \le C\|\nabla_{tan} u\|_2$.
It follows that
\begin{equation}\label{R-3}
\|\left(\frac{\partial u}{\partial\nu}\right)_\pm \|_2
\le C\| \nabla_{tan} u\|_2
\le C \|(\nabla u)_\mp\|_2\\
 \le C
\|\left(\frac{\partial u}{\partial\nu}\right)_\mp \|_2.
\end{equation}
Consequently, by the jump relation, for any $f\in L^2_0(\partial\Omega, \mathbb{R}^m)$,
\begin{equation}\label{R-4}
\aligned
\| f\|_2  &\le  \|\left(\frac{\partial u}{\partial\nu}\right)_+ \|_2
+
\|\left(\frac{\partial u}{\partial\nu}\right)_- \|_2
\le C \|\left(\frac{\partial u}{\partial\nu}\right)_\pm \|_2\\
& =C \| \big(\pm (1/2) I +\mathcal{K}_A \big) f\|_2.
\endaligned
\end{equation}
  Furthermore, if $f\in L^2(\partial\Omega, \mathbb{R}^m)$ and $g=f-f_{\partial\Omega}$,
then
\begin{equation}\label{R-6}
\aligned
\| f\|_2 & \le C \| (-(1/2)I +\mathcal{K}_A ) g\|_2 +\| f_{\partial\Omega}\|_2\\
& \le C  \| (-(1/2)I +\mathcal{K}_A ) f\|_2 + C \| f_{\partial\Omega}\|_2\\
& \le C  \| (-(1/2)I +\mathcal{K}_A ) f\|_2.
\endaligned
\end{equation}
We remark that the last inequality in (\ref{R-6})
follows from the observation that
$f_{\partial\Omega}$ is also the mean value of $-(\frac{\partial u}{\partial\nu})_-$
on $\partial\Omega$. Since $\left(\frac{\partial u}{\partial \nu}\right)_+$
has mean value zero, this is
a simple consequence of the jump relation (\ref{jump-relation}).

Thus, to complete the proof, we only need to show that the operators
$(1/2)I +\mathcal{K}_A: L^2_0(\partial\Omega, \mathbb{R}^m)
\to L^2_0(\partial\Omega, \mathbb{R}^m)$
and $-(1/2)I +\mathcal{K}_A: L^2(\partial\Omega, \mathbb{R}^m)
\to L^2(\partial\Omega, \mathbb{R}^m)$ are onto.
To this end, we consider a family of matrices $ A^s = sA +(1-s) I$, where $0\le s\le 1$.
Note that by \cite{Verchota-1984}, $\pm (1/2) I +\mathcal{K}_{A^0}$
are invertible on $L_0^2(\partial\Omega, \mathbb{R}^m)$ and
$L^2(\partial\Omega, \mathbb{R}^m)$ respectively.
Also observe that for each $s\in [0,1]$, the matrix $A^s$ satisfies the same 
conditions as  $A$. Hence, 
\begin{equation}\label{R-7}
\aligned
& 
\| f\|_2 \le C \| ((1/2)I +\mathcal{K}_{A^s}) f\|_2 \qquad\qquad \text{ for any } f
\in L^2_0(\partial\Omega, \mathbb{R}^m),\\
& \| f\|_2 \le C \| (-(1/2)I +\mathcal{K}_{A^s}) f\|_2 \qquad \quad
\text{ for any } f
\in L^2(\partial\Omega, \mathbb{R}^m),
\endaligned
\end{equation}
where $C$ is independent of $s$.
Since $\| A^{s_1} -A^{s_2}\|_{C^{\lambda}(\mathbb{R}^d)}
\le |s_1-s_2|\| A\|_{C^{\lambda}(\mathbb{R}^d)}$,
it follows from Theorem \ref{operator-approximation-theorem} that
$\{ (1/2)I +\mathcal{K}_{A^s}: 0\le s\le 1\}$ and
$\{ -(1/2)I +\mathcal{K}_{A^s}: 0\le s\le 1\}$
are continuous families of bounded operators 
on $L_0^2(\partial\Omega,\mathbb{R}^m)$ and  
 $L^2(\partial\Omega,\mathbb{R}^m)$ respectively.
This, together with the estimates in  (\ref{R-7}) and the invertibility results
for $s=0$, gives the desired invertibility for $s=1$.
The operator norm estimates in (\ref{inverse-bound})
follow directly from (\ref{R-4}) and (\ref{R-6}). 
\end{proof}

\begin{remark}\label{single-layer-invertibility-remark-1}
{\rm
Under the same assumptions on $A$ and $\Omega$ as in Lemma \ref{invertibility-of-Neumann-operator},
the operator $\mathcal{S}_A: L^2(\partial\Omega, \mathbb{R}^m)
\to W^{1,2}(\partial\Omega, \mathbb{R}^m)$ is invertible
and $\| \big( \mathcal{S}_A\big)^{-1}\|_{W^{1,2}\to L^2} \le C$.
To see this, we let 
$f\in L^2(\partial\Omega, \mathbb{R}^m)$ 
and $u=\mathcal{S}(f)$. It follows from the proof of
Lemma \ref{invertibility-of-Neumann-operator} that
\begin{equation}\label{R-8}
\| (\nabla u)_-\|_2 \le C\|\nabla_{tan} u\|_2 +Cr_0^{-1}\| u\|_2.
\end{equation}
This, together with $\|(\nabla u)_+\|_2 \le C\|\nabla_{tan} u\|_2$ and
the jump relation, gives
\begin{equation}\label{R-9}
\| f\|_2 \le C \|\nabla_{tan} \mathcal{S}(f)\|_2 +C r_0^{-1} \| \mathcal{S}(f)\|_2
\le C \| \mathcal{S}(f)\|_{1,2}.
\end{equation}
Estimate (\ref{R-9}) implies that $\mathcal{S}: L^2(\partial\Omega, \mathbb{R}^m)
\to W^{1,2}(\partial\Omega, \mathbb{R}^m)$ is one-to-one.
A continuity argument similar to that in the proof of Lemma \ref{invertibility-of-Neumann-operator}
shows that the operator is in fact invertible.
}
\end{remark}

\begin{remark}\label{operator-invertiblity-remark}
{\rm
Under the same conditions on $A$ as in Theorem \ref{Rellich-imply-solvability},
the estimates in (\ref{inverse-bound}) and (\ref{R-9}) hold with a ``good'' constant $C$.
}
\end{remark}

We are now in a position to give the proof of Theorems \ref{local-Rellich-imply-solvability}
and \ref{Rellich-imply-solvability}.

\noindent{\bf Proof of Theorems \ref{local-Rellich-imply-solvability} and 
\ref{Rellich-imply-solvability}.}

As we mentioned earlier, the uniqueness for the $L^2$ Neumann and regularity problems
follows from the Green's identity (\ref{Green-identity}) by approximating $\Omega$
from inside.
The existence for the $L^2$ Neumann and regularity problems is a direct consequence
of the invertibility of $(1/2)I +\mathcal{K}_A$ on $L^2_0(\partial\Omega, \mathbb{R}^m)$
and that of $\mathcal{S}_A: L^2(\partial\Omega,\mathbb{R}^m)\to
W^{1,2}(\partial\Omega, \mathbb{R}^m)$ respectively.
Since $-(1/2)I +\mathcal{K}_A$ is invertible on $L^2(\partial\Omega, \mathbb{R}^m)$,
it follows by duality that $-(1/2)I +\mathcal{K}_A^*$ is also invertible
on $L^2(\partial\Omega, \mathbb{R}^m)$ and $\|(-(1/2)I +\mathcal{K}^*_A)^{-1}\|_{L^2\to L^2}
=\|(-(1/2)I +\mathcal{K}_A)^{-1}\|_{L^2\to L^2}$.
This gives the existence for the $L^2$ Dirichlet problem in $\Omega$.
Note that under the conditions in Theorem \ref{Rellich-imply-solvability},
the operator norms of $(\pm (1/2)I +\mathcal{K}_A)^{-1}$ and $(\mathcal{S}_A)^{-1}$
are bounded by a ``good'' constant $C$.
It follows that estimates in (\ref{Neumann-regularity-estimate})
and $\|(u)^*\|_2 \le C \| u\|_2$
hold with a ``good'' constant $C$.

To establish the uniqueness, we construct a matrix of Green's functions 
$(G^{\alpha\beta} (X,Y))$ for $\Omega$, where
\begin{equation}\label{Green's-function}
G^{\alpha\beta} (X,Y)=\Gamma^{\alpha\beta}(X,Y)-W^{\alpha\beta} (X,Y)
\end{equation}
and for each $\beta$ and $Y\in \Omega$, $W^\beta (\cdot,Y)=
(W^{1\beta}(\cdot,Y), \dots, W^{m \beta}(\cdot,Y))$ is the solution to the 
$L^2$ regularity problem for $\mathcal{L}(u)=0$ in $\Omega$
with boundary data 
$$
\Gamma^\beta (\cdot, Y)=
(\Gamma^{1\beta}(\cdot, Y), \dots, \Gamma^{m\beta} (\cdot, Y))
\qquad \text{ on } \partial\Omega.
$$

Suppose now that $\mathcal{L}(u)=0$ in $\Omega$, $(u)^*\in L^2(\partial\Omega)$
and $u=0$ n.t. on $\partial\Omega$.
For $\rho>0$ small, choose $\varphi=\varphi_\rho$ so that
$\varphi =1 $ in $\{ X\in \Omega: \text{ dist}(X,\partial\Omega)\ge 2\rho\}$,
$\varphi=0$ in $\{ X\in \Omega: \text{ dist}(X, \partial\Omega) \le \rho\}$
and $|\nabla \varphi|\le C\rho^{-1}$.
Fix $Y\in \Omega$ so that dist$(Y, \partial\Omega)\ge 2\rho$.
It follows from (\ref{fundamental-solution-representation}) that
\begin{equation}\label{R-10}
\aligned
u^\gamma (Y)  &=u^\gamma (Y)\varphi (Y)
=\int_\Omega
a_{ij}^{\alpha\beta} (X)
\frac{\partial}{\partial x_j}
\big\{ G^{\beta\gamma} (X,Y)\big\}
\frac{\partial}{\partial x_i} (u^{\alpha} \varphi)\, dX\\
&
=-\int_\Omega a_{ij}^{\alpha\beta} (X) G^{\beta\gamma} (X,Y) \frac{\partial u^\alpha}
{\partial x_i} \cdot \frac{\partial \varphi}{\partial x_j}\, dX\\
&\qquad\qquad
+\int_\Omega a_{ij}^{\alpha\beta} 
\frac{\partial}{\partial x_j} \big\{ G^{\beta\gamma}(X,Y)\big\}
 u^\alpha \frac{\partial\varphi}
{\partial x_i}\, dX,
\endaligned
\end{equation}
where we have used the integration by parts and $A^*=A$.
This gives
\begin{equation}\label{R-11}
|u(Y)|
 \le \frac{C}{\rho} \int_{F_\rho} |G(X,Y)||\nabla u|\, dX
+\frac{C}{\rho}
\int_{F_\rho} |\nabla_X G(X,Y)||u|\, dX,
\end{equation}
where $F_\rho=\{ X\in \Omega: \rho \le \text{dist}(X, \partial\Omega)\le 2\rho\}$.
Using $G(\cdot, Y)=u=0$ n.t. on $\partial\Omega$ as well as 
the gradient estimate (\ref{gradient-estimate}) on $u$,
we may deduce from (\ref{R-11}) that
\begin{equation}\label{R-12}
|u(Y)|\le C\int_{\partial\Omega} (\nabla G(\cdot, Y))^*_{3\rho}
(u)^*_{3\rho} \, d\sigma,
\end{equation}
where $(u)^*_{3\rho} (P) =\sup\{ |\nabla u(X)|: X\in \gamma (P) 
\text{ and } \text{dist}(X, \partial\Omega)<3\rho\}$.
As
$(\nabla G(\cdot, Y))^*_{3\rho}
(u)^*_{3\rho}\in L^1 (\partial\Omega)$,
 we may conclude
from (\ref{R-12}) by the Lebesgue dominated convergence theorem that
$u(Y)=0$.
This completes the proof.
\qed

\section{Solvability for small scales, Part I}

The main purpose of this and next sections is to establish the following theorem.

\begin{thm}\label{local-solvability}
Let $A=(a_{ij}^{\alpha\beta})$ be a real matrix satisfying the symmetry
condition (\ref{symmetry}), the ellipticity condition (\ref{ellipticity})
and the smoothness condition (\ref{smoothness}).
Let $R\ge 1$.
Then for any bounded Lipschitz domain $\Omega$ with connected boundary and
diam$(\Omega)\le R$, the $L^2$ Neumann and regularity problems for
$\text{div}(A\nabla u)=0$ in $\Omega$ are uniquely solvable and
the solutions satisfy the estimates in (\ref{Neumann-regularity-estimate})
with constant $C$ depending only on $\mu$, $\lambda$, $\tau$,
the Lipschitz character of $\Omega$, and $R$ (if $\text{diam}(\Omega)>1$).
Furthermore, the $L^2$ Dirichlet problem for
$\text{div}(A\nabla u)=0$ in $\Omega$ is uniquely solvable
with estimate $\|(u)^*\|_2 \le C\| u\|_2$.
\end{thm}

\begin{remark}\label{periodicity-remark}
{\rm
Note that the periodicity of $A$ is not needed in Theorem \ref{local-solvability}.
This is because we may reduce the general case to the case of the periodic 
coefficients. 
Indeed, by translation, we may assume that $0\in \Omega$.
If diam$(\Omega)\le (1/4)$, we construct $\widetilde{A}\in 
\Lambda (\mu, \lambda, \tau_0)$ so that
$\widetilde{A}=A$ on $[-3/8,3/8]^d$, where $\tau_0$ depends on 
$\mu$ and $\tau$.
The boundary value problems for
$\text{div}(\widetilde{A} \nabla u)=0$ in $\Omega$
are the same as those for $\text{div}(A\nabla u)=0$ in $\Omega$.
Suppose now that $r_0=\text{diam}(\Omega)> (1/4)$. 
By rescaling, the boundary value problems
for $\text{div}(A\nabla u)=0$ in $\Omega$
are equivalent to that of
$\text{div}(A^1\nabla u)=0$ in $\Omega_1$,
where $A^1(X)=A(4r_0X)$ and 
$\Omega_1 =\{ X\in \mathbb{R}^d: 4r_0X\in \Omega\}$.
Since $\text{diam}(\Omega_1)= (1/4)$, we have reduced the case to the
previous one.
}
\end{remark}

By Remark \ref{periodicity-remark} it is enough to prove Theorem \ref{local-solvability}
under the additional assumption that $\text{diam}(\Omega)\le (1/4)$
and $A$ is periodic with respect to $\mathbb{Z}^d$ (thus $A\in \Lambda(\mu, \lambda, \tau)$).
Furthermore, in view of Theorem \ref{local-Rellich-imply-solvability},
it suffices to show that if $\mathcal{L}=-\text{div}(A\nabla )$
with $A\in \Lambda(\mu,\lambda, \tau)$ and $A^*=A$ and if
$\Omega$ is a Lipschitz domain with $\text{diam}(\Omega)\le (1/4)$, then
$\mathcal{L}$ has the Rellich property in $\Omega$
with constant $C(\Omega)$ depending only on $\mu$, $\lambda$, $\tau$
and the Lipschitz character of $\Omega$.
We point out that if $A$ is Lipschitz continuous, 
the Rellich property follows readily from the Rellich type identities,
as in case of constant coefficients (see e.g. \cite{fabes2}).
However, since it is essential to us that
the constant $C(\Omega)$ depends only on the Lipschitz character of $\Omega$,
 the proof for operators with H\"older continuous coefficients is quite
involved.

As we pointed out above,
Theorem \ref{local-solvability} is a consequence of the following.

\begin{thm}\label{Holder-continuous-Rellich}
Let $\mathcal{L}=-\text{div}(A\nabla u)$ with $A\in\Lambda(\mu, \lambda,\tau)$
and $A^*=A$.
Let $\Omega$ be a bounded Lipschitz domain with $\text{diam}(\Omega)\le (1/4)$
and connected boundary.
Then $\mathcal{L}$ has the Rellich property
in $\Omega$ with a ``good'' constant.
\end{thm}

By translation we may assume that $0\in \Omega$ and thus $\Omega\subset [-1/4,1/4]^d$.
We divide the proof of Theorem \ref{Holder-continuous-Rellich} into three steps.

\noindent {\it Step One}: Establish the invertibility of $\pm (1/2)I +\mathcal{K}_A$
under the additional assumption that
\begin{equation}\label{additional-assumption-1}
\left\{ \aligned
& A \in C^1 ([-1/2, 1/2]^d\setminus \partial\Omega),\\
& |\nabla A(X)|\le C_1  \big\{ \text{dist}(X, \partial\Omega)\big\}^{\lambda_0-1}
\text{ for any } X\in [-1/2, 1/2]^d \setminus \partial\Omega,
\endaligned
\right.
\end{equation}
where $\lambda_0\in (0,1)$.

Clearly, if $A\in C^1([-1/2,1/2]^d)$, then it satisfies (\ref{additional-assumption-1}).

\begin{lemma}\label{lemma-5.1}
Let $\Omega$ be a bounded Lipschitz domain with connected boundary.
Suppose that $0\in \Omega$ and $r_0=\text{diam}(\Omega)\le (1/4)$.
Let $A\in \Lambda(\mu, \lambda,\tau)$ be such that $A^*=A$ and condition
(\ref{additional-assumption-1}) holds.
Assume that $\mathcal{L}(u)=0$ in $\Omega$,
$(\nabla u)^*\in L^2 (\partial\Omega)$ and $(\nabla u)_+$ exists
n.t. on $\partial\Omega$.
Then
\begin{equation}\label{local-interior-estimate}
\aligned
& \int_{\partial\Omega} |\nabla u|^2\, d\sigma
\le C\int_{\partial\Omega} |\frac{\partial u}{\partial\nu}|^2\, d\sigma
+C\int_\Omega (|\nabla A|+r_0^{-1}) |\nabla u|^2\, dX,\\
& \int_{\partial\Omega} |\nabla u|^2\, d\sigma
\le C\int_{\partial\Omega} |\nabla_{tan} u|^2\, d\sigma
+C\int_\Omega (|\nabla A|+r_0^{-1}) |\nabla u|^2\, dX,
\endaligned
\end{equation}
where $C$ depends only on $\mu$ and  the Lipschitz character of $\Omega$.
\end{lemma}

\begin{proof}
Let $\mathbf{h}$ be a $C^1$ vector field on $\mathbb{R}^d$
such that supp$(\mathbf{h})\subset \{X: \text{dist}(X, \partial\Omega)<c r_0\}$,
$|\nabla \mathbf{h}|\le Cr_0^{-1}$ and
$<\mathbf{h}, n> \ge c>0$ on $\partial\Omega$.
As in the case of constant coefficients,
the estimates in (\ref{local-interior-estimate}) follow from the so-called
Rellich identities,
\begin{equation}\label{local-Rellich}
\aligned
& \int_{\partial\Omega}
<\mathbf{h}, n> a_{ij}^{\alpha\beta} \frac{\partial u^\alpha}{\partial x_i}
\cdot \frac{\partial u^\beta}{\partial x_j}\, d\sigma
=2 \int_{\partial\Omega} <\mathbf{h}, \nabla u^\alpha> 
\left(\frac{\partial u}{\partial \nu}\right)^\alpha\, d\sigma + I_1,\\
&
\int_{\partial\Omega}
<\mathbf{h}, n> a_{ij}^{\alpha\beta} \frac{\partial u^\alpha}{\partial x_i}
\cdot \frac{\partial u^\beta}{\partial x_j}\, d\sigma
=2\int_{\partial\Omega}
h_k a_{ij}^{\alpha \beta} \frac{\partial u^\beta}{\partial x_j}
\left(n_k \frac{\partial}{\partial x_i}
-n_i \frac{\partial}{\partial x_k}\right) u^\alpha\, d\sigma
+I_2
\endaligned
\end{equation}
where 
$$
|I_1| +|I_2|
\le C \int_{\Omega}\big\{
 |\nabla\mathbf{h}| + |\mathbf{h}| |\nabla A|\big\} |\nabla u|^2\, dX
$$
and $C$ depends only on $\mu$.
The proof of (\ref{local-Rellich}), which uses integration by parts and
the assumption that $A^*=A$, is similar to the case of constant coefficients. 
The latter may be found in \cite{fabes2}.
\end{proof}

\begin{remark}\label{local-exterior-remark}
{\rm 
Let $\mathcal{L}(u)=0$ in $(-1/2,1/2)^d\setminus \overline{\Omega}$.
Suppose that $(\nabla u)^*\in L^2(\partial\Omega)$ and $(\nabla u)_-$ exists
n.t. on $\partial\Omega$. Under the same conditions on $\Omega$ and $A$ as in
Lemma \ref{lemma-5.1}, we have
\begin{equation}\label{local-exterior-estimate}
\aligned
& \int_{\partial\Omega} |(\nabla u)_-|^2\, d\sigma
\le C\int_{\partial\Omega} |\big(\frac{\partial u}{\partial\nu}\big)_-|^2\, d\sigma
+C\int_{\Omega_-\cap [-1/2,1/2]^d}
 (|\nabla A|+r_0^{-1}) |\nabla u|^2\, dX,\\
& \int_{\partial\Omega} |(\nabla u)_-|^2\, d\sigma
\le C\int_{\partial\Omega} |(\nabla_{tan} u)_-|^2\, d\sigma
+C\int_{\Omega_-\cap [-(1/2),1/2]^d} (|\nabla A|+r_0^{-1}) |\nabla u|^2\, dX,
\endaligned
\end{equation}
where $C$ depends only on $\mu$ and the Lipschitz character of $\Omega$.
The proof is similar to that of Lemma \ref{lemma-5.1}.
}
\end{remark}

\begin{lemma}\label{local-interior-lemma-1}
Under the same assumptions as in Lemma \ref{lemma-5.1}, we have
\begin{equation}\label{local-interior-estimate-1}
\aligned
\int_{\partial\Omega}
|\nabla u|^2\, d\sigma 
& \le C \big\{ 1+ r_0^{2\lambda_0} \rho^{2\lambda_0-2}\big\}
\int_{\partial\Omega}
\big|\frac{\partial u}{\partial\nu}\big|^2\, d\sigma
+C (\rho r_0 )^{\lambda_0} \int_{\partial\Omega} |(\nabla u)^*|^2\, d\sigma,\\
\int_{\partial\Omega}
|\nabla u|^2\, d\sigma 
& \le C \big\{ 1+ r_0^{2\lambda_0} \rho^{2\lambda_0-2}\big\}
\int_{\partial\Omega}
\big|\nabla_{tan} u|^2\, d\sigma
+C (\rho r_0)^{\lambda_0} \int_{\partial\Omega} |(\nabla u)^*|^2\, d\sigma,
\endaligned
\end{equation}
where $0<\rho<1$ and $C$ depends only on $\mu$, 
the Lipschitz character of $\Omega$ and
$\lambda_0$, $C_1$ in (\ref{additional-assumption-1}).
\end{lemma}

\begin{proof} Write $\Omega=F_1 \cup F_2$, 
where $F_1=\{ X\in \Omega: \text{dist}(X,\partial\Omega)\le \rho r_0\}$
and $F_2=\{ X \in \Omega: \text{dist}(X,\partial\Omega)>\rho r_0\}$.
Using the condition (\ref{additional-assumption-1}), we obtain
\begin{equation}
\aligned
\int_\Omega |\nabla A ||\nabla u|^2\, dX
&\le C_1 \int_{F_1} \{ \text{dist}(X,\partial\Omega)\}^{\lambda_0-1} |\nabla u|^2\, dX
+C_1 (\rho r_0)^{\lambda_0-1} \int_{F_2} |\nabla u|^2\, dX\\
& \le C(\rho r_0)^{\lambda_0} \int_{\partial\Omega} |(\nabla u)^*|^2\, d\sigma
+C_1 (\rho r_0)^{\lambda_0-1} \int_\Omega |\nabla u|^2\, dX.
\endaligned
\end{equation}
This, together with (\ref{local-interior-estimate}) and (\ref{Green-estimate})
for $\Omega_+$,
gives
\begin{equation}\label{I-5}
\|\nabla u\|_2^2
\le C\|\frac{\partial u}{\partial\nu}\|_2^2
+C (1+r_0^{\lambda_0} \rho^{\lambda_0-1})\|\frac{\partial u}{\partial\nu}\|_2
\|\nabla_{tan} u\|_2
+C (\rho r_0)^{\lambda_0} \|(\nabla u)^*\|_2^2.
\end{equation}
The first inequality in (\ref{local-interior-estimate-1})
follows from (\ref{I-5}) by the Cauchy inequality with an $\varep$.
The proof of the second inequality in (\ref{local-interior-estimate-1})
is similar.
\end{proof}

\begin{remark}
{\rm
Let $\mathcal{L}(u)=0$ in $\Omega_-$.
Suppose that $(\nabla u)^*\in L^2 (\partial\Omega)$, $(\nabla u)_-$
exists n.t. on $\partial\Omega$, and
$|u(X)|=O(|X|^{2-d})$ as $|X|\to\infty$.
In view of Remark \ref{local-exterior-remark} and (\ref{Green-estimate})
for $\Omega_-$, 
the same argument as in the proof of 
Lemma \ref{local-interior-lemma-1} shows that
\begin{equation}\label{local-exterior-estimate-1}
\aligned
\int_{\partial\Omega}|(\nabla u)_-|^2 \, d\sigma 
& \le C \big\{ 1+ r_0^{2\lambda_0} \rho^{2\lambda_0-2}\big\}
\| \big( \frac{\partial u}{\partial\nu}\big)_-\|_2^2
+C (\rho r_0 )^{\lambda_0} \|(\nabla u)^*\|_2^2\\
& \qquad\qquad \qquad 
+ C (\rho r_0)^{\lambda_0 -1} |u_{\partial\Omega}|
\left| \int_{\partial \Omega} \left( \frac{\partial u}{\partial \nu}\right)_-\, d\sigma\right|,\\
\int_{\partial\Omega}|(\nabla u)_-|^2\, d\sigma 
& \le C \big\{ 1+ r_0^{2\lambda_0} \rho^{2\lambda_0-2}\big\}
\| \big( \nabla_{tan} u)_-\|_2^2
+C (\rho r_0 )^{\lambda_0} \|(\nabla u)^*\|_2^2\\
& \qquad\qquad\qquad
+ C (\rho r_0)^{\lambda_0 -1} |u_{\partial\Omega}|
\left| \int_{\partial \Omega} \left( \frac{\partial u}{\partial \nu}\right)_-\, d\sigma\right|,
\endaligned
\end{equation}
for any $0<\rho<1$.
}
\end{remark}

The following theorem completes Step One.

\begin{thm}{\label{step-one-theorem}}
Suppose that $\Omega$ and $A$ satisfy the conditions 
in Theorem \ref{Holder-continuous-Rellich}.
We further assume that $0\in \Omega$ and
$A$ satisfies (\ref{additional-assumption-1}).
Then $(1/2)I +\mathcal{K}_A$ and $-(1/2)I +\mathcal{K}_A $ are invertible
on $L^2_0(\partial\Omega, \mathbb{R}^m)$ and $L^2(\partial\Omega, \mathbb{R}^d)$
respectively, and the estimates in (\ref{inverse-bound}) 
hold with a constant $C$ depending only on
$\mu$, $\lambda$, $\tau$, the Lipschitz character of $\Omega$
 and $C_1$, $\lambda_0$ in (\ref{additional-assumption-1}).
\end{thm}

\begin{proof}
Let $f\in L^2_0(\partial\Omega, \mathbb{R}^m)$ and $u=\mathcal{S}(f)$ be the single
layer potential. 
In view of (\ref{local-exterior-estimate-1}), we obtain
\begin{equation}\label{I-5.6}
\| (\nabla u)_-\|_2 \le \rho_1^{\lambda_0-1} \|(\nabla u)_+\|_2 +C\rho_1^{\lambda_0/2}
\| f\|_2,
\end{equation}
for any $0<\rho_1<1$,
where we also used the fact that $(\nabla_{tan} u)_- =(\nabla_{tan} u)_+$ and
$\|(\nabla u)^*\|_2 \le C\| f\|_2$.
Similarly, by (\ref{local-interior-estimate-1}),
\begin{equation}\label{I-5.7}
\|(\nabla u)_+\|_2 
\le C\rho_2^{\lambda_0-1} \| \big(\frac{\partial u}{\partial \nu}\big)_+\|_2
+C\rho_2^{\lambda_0/2} \| f\|_2,
\end{equation}
for any $0<\rho_2<1$. It follows from the jump relation (\ref{jump-relation}),
(\ref{I-5.6}) and (\ref{I-5.7}) that
\begin{equation}\label{I-5.8}
\aligned
\| f\|_2
& \le \| \big(\frac{\partial u}{\partial \nu}\big)_+\|_2
+\| \big(\frac{\partial u}{\partial \nu}\big)_-\|_2\\
& \le C \rho_1^{\lambda_0-1}\rho_2^{\lambda_0-1}
\| \big(\frac{\partial u}{\partial \nu}\big)_+\|_2
+ C\big\{ \rho_1^{\lambda_0-1}\rho_2 ^{\lambda_0/2} +\rho_1^{\lambda_2/2}\big\}
\| f\|_2.
\endaligned
\end{equation}
We now choose $\rho_1\in (0,1)$ and then $\rho_2\in (0,1)$ so that
$C\big\{ \rho_1^{\lambda_0-1}\rho_2 ^{\lambda_0/2} +\rho_1^{\lambda_2/2}\big\}
\le (1/2)$.
This gives
\begin{equation}\label{I-5.9}
\| f\|_2 \le C \| \big(\frac{\partial u}{\partial \nu}\big)_+\|_2
=C \| ((1/2)I +\mathcal{K}_A ) f\|_2,
\end{equation}
for any $f\in L^2_0(\partial\Omega, \mathbb{R}^m)$.
The same argument also shows that for any $f\in L^2_0(\partial\Omega, \mathbb{R}^m)$,
\begin{equation}\label{I-5.10}
\| f\| \le C \| \big(\frac{\partial u}{\partial \nu}\big)_-\|_2
=C \| (-(1/2)I +\mathcal{K}_A ) f\|_2.
\end{equation}
The rest of the proof is the same as that of Lemma \ref{invertibility-of-Neumann-operator}.
\end{proof}

\begin{remark}\label{local-single-layer-invertibility-remark-1}
{\rm
Let $f\in L^2(\partial\Omega)$ and $u=\mathcal{S}(f)$. It follows from 
(\ref{local-interior-estimate-1}) and (\ref{local-exterior-estimate-1}) that
\begin{equation}\label{I-5.13}
\aligned
\| (\nabla u)_+\|_2  & \le C\rho_1^{\lambda_0 -1} \|\nabla_{tan} u\|_2
+C\rho_1^{\lambda_0/2} \| f\|_2,\\
\| (\nabla u)_-\|_2
&\le C \rho_2^{\lambda_0-1} \|\nabla_{tan} u\|_2 +C \rho_2^{\lambda_0/2} \| f\|_2
+ C r_0^{-1}\| u\|_2,
\endaligned
\end{equation}
for any $\rho_1, \rho_2 \in (0,1)$. This, together with the jump relation,
implies
\begin{equation}\label{I-5.14}
\| f\|_2 \le C \|\nabla_{tan} \mathcal{S}(f)\|_2 +C r_0^{-1} \| \mathcal{S}(f)\|_2
\le C \| \mathcal{S}(f)\|_{1,2}.
\end{equation}
Thus $\mathcal{S}: L^2(\partial\Omega, \mathbb{R}^m)
\to W^{1,2}(\partial\Omega, \mathbb{R}^m)$ is
one-to-one.
A continuity argument similar to that in the proof of Lemma \ref{invertibility-of-Neumann-operator}
shows that the operator is in fact invertible.
}
\end{remark}

\section{Solvability for small scales, Part II}

In this section we complete the second and third steps in
the proof of Theorems \ref{Holder-continuous-Rellich}.

\noindent{\it Step Two}: Given any $A\in \Lambda(\mu, \lambda, \tau)$
and $\Omega$ such that $A^*=A$, $0\in \Omega$ and $r_0=$diam$(\Omega)\le (1/4)$,
construct $\widetilde{A}\in \Lambda(\mu, \lambda_0, \tau_0)$ 
with $\lambda_0$ and $\tau_0$ depending only
on $\mu$, $\lambda$, $\tau$ and the Lipschitz character of $\Omega$,
such that
\begin{equation}\label{part-2.1}
\widetilde{A}(X)= A(X) \qquad \text{ if } \text{dist}(X, \partial \Omega)
\le cr_0,
\end{equation}
and such that the operators
\begin{equation}\label{part-2.2}
\aligned
 (1/2)I +\mathcal{K}_{\widetilde{A}}: 
L_0^2(\partial\Omega, \mathbb{R}^m) & \to L^2_0(\partial\Omega, \mathbb{R}^m),\\
 -(1/2)I +\mathcal{K}_{\widetilde{A}}: 
L^2(\partial\Omega, \mathbb{R}^m) & \to L^2(\partial\Omega, \mathbb{R}^m),\\
 \mathcal{S}_{\widetilde{A}}: L^2(\partial\Omega, \mathbb{R}^m)
& \to W^{1,2}(\partial\Omega, \mathbb{R}^m)
\endaligned
\end{equation}
are invertible and the operator norms of their inverses are
bounded by a ``good'' constant.

\begin{lemma}\label{construction-A-bar-lemma}
Given $A\in \Lambda (\mu, \lambda, \tau)$ and a Lipschitz domain $\Omega$
such that diam$(\Omega)\le (1/4)$ and $0\in \Omega$.
There exists $\bar{A}\in \Lambda(\mu, \lambda_0, \tau_0)$
such that $\bar{A}=A$ on $\partial\Omega$ and $\bar{A}$
satisfies the condition (\ref{additional-assumption-1}),
where $\lambda_0\in (0,\lambda]$, $\tau_0$ and $C_1$ in (\ref{additional-assumption-1})
depend only on $\mu$, $\lambda$, $\tau$ and the Lipschitz
character of $\Omega$.
In addition, $(\bar{A})^* =\bar{A}$ if $A^*=A$.
\end{lemma}

\begin{proof}
By periodicity it suffices to define
$\bar{A}=(\bar{a}_{ij}^{\alpha\beta})$ 
on $[-1/2,1/2]^d$. This is done as follows.
On $\Omega$ we define $\bar{A}$ to be the Poisson extension of  $A$ on 
$\partial\Omega$; i.e., 
$\bar{a}_{ij}^{\alpha\beta}$
is harmonic in $\Omega$ and $\bar{a}_{ij}^{\alpha\beta}
=a_{ij}^{\alpha\beta}$ on $\partial\Omega$,
 for each $i,j,\alpha,\beta$.
On $[-1/2,1/2]^d\setminus \overline{\Omega}$, we define
$\bar{A}$ to be the harmonic function in 
$(-1/2,1/2)^d\setminus \overline{\Omega}$
with boundary data $\bar{A}=A$ on $\partial\Omega$ and
$\bar{A}=I$ on $\partial [-1/2,1/2]^d$.
Note that the latter boundary condition allows us to extend $\bar{A}$
to $\mathbb{R}^d$ by periodicity.

Since $\bar{a}_{ij}^{\alpha\beta}\xi_i^\alpha\xi_j^\beta$ is harmonic in $(-1/2,1/2)^d
\setminus \partial\Omega$, the ellipticity
condition (\ref{ellipticity})
for $\bar{A}$ follows readily from the maximum principle.
By the solvability of Laplace's equation in Lipschitz domains
with H\"older continuous data (see e.g. \cite{Kenig-1994}),
 there exists $\lambda_1\in (0,1)$, depending only on
the Lipschitz character of $\Omega$, such that
$\bar{A}\in C^{\lambda_0}(\overline{\Omega})$
and $\bar{A}\in C^{\lambda_0} ([-1/2,1/2]^d\setminus \Omega)$,
where $\lambda_0=\lambda$ if $\lambda<\lambda_1$, and $\lambda_0=\lambda_1$
if $\lambda\ge \lambda_1$.
It follows that $\bar{A}\in C^{\lambda_0}(\mathbb{R}^d)$.
Using the well known interior estimates for harmonic functions, one may also show that
$|\nabla \bar{A}(X)|\le C_1 \big\{ \text{dist} (X, \partial\Omega)\big\}^{\lambda_0-1}$
for $X\in [-3/4,3/4]^d\setminus \partial\Omega$,
where $C_1$ depends only on $\mu$, $\lambda$, $\tau$ and the Lipschitz
character of $\Omega$.
Thus we have proved that $\bar{A}\in \Lambda (\mu, \lambda_0, \tau)$
and satisfies the condition (\ref{additional-assumption-1}).
Clearly, $(\bar{A})^* =\bar{A}$ if $A^*=A$.
\end{proof}

Let $\theta\in C_0^\infty (-1/2,1/2)$ such that $0\le \theta\le 1$
and $\theta=1$ on $(-1/4,1/4)$.
Given $A\in \Lambda(\mu, \lambda, \tau)$ with $A^*=A$, 
define
\begin{equation}\label{approximation-3}
A^\rho (X) =\theta \left(\frac{\delta(X)}{\rho}\right) A (X)
+ \left[ 1- \theta \left(\frac{\delta(X)}{\rho}\right)\right] \bar{A}(X)
\end{equation}
for $X\in [-1/2,1/2]^d$, 
where $\rho\in (0,1/8)$, $\delta(X)=\text{dist}(X, \partial\Omega)$ and
$\bar{A} (X)$ is the matrix constructed in Lemma \ref{construction-A-bar-lemma}.
Extend $A^\rho$ to $\mathbb{R}^d$ by periodicity.
Clearly, $A^\rho$ satisfies the ellipticity condition (\ref{ellipticity})
and $(A^\rho)^* =A^\rho$.

\begin{lemma}\label{lemma-6.2}
Let $A^\rho$ be defined by (\ref{approximation-3}). Then 
\begin{equation}\label{I-6.2}
 \| A^\rho -\bar{A}\|_\infty \le C\rho^{\lambda_0} \quad \text{ and } \quad
 \| A^\rho-\bar{A}\|_{C^{0, \lambda_0}(\mathbb{R}^d)}
\le C,
\end{equation}
where $C$ depends only on $\mu$, $\lambda$, $\tau$ and the Lipschitz character of $\Omega$.
\end{lemma}

\begin{proof}
Let $H^\rho =A^\rho-\bar{A}$.
Given $X\in [-1/2,1/2]^d$, let $P\in \partial\Omega$ such that
$|X-P|=\delta(X)$. Since $A(P)=\bar{A}(P)$, we have
$$
|A(X)-\bar{A}(X)|
\le |A(X)-A(P)| +|\bar{A}(P)-\bar{A}(X)|
\le C |X-P|^{\lambda_0}=C \big\{ \delta(X)\big\}^{\lambda_0}.
$$
It follows that
$$
|H^\rho(X)| \le C \theta (\rho^{-1}\delta(X)) \big\{ \delta(X)\big\} ^{\lambda_0}
= C\theta (\rho^{-1} \delta(X)) \big\{ \rho^{-1} \delta(X)\big\}^{\lambda_0} \rho^{\lambda_0}
\le C \rho^{\lambda_0}.
$$
This gives $\| A^\rho -A\|_\infty \le C\rho^{\lambda_0}$.

Next we show $|H^\rho (X)-H^\rho (Y)|\le C |X-Y|^{\lambda_0}$ for any $X,Y\in \mathbb{R}^d$.
Since $\|H^\rho\|_\infty \le C\rho^{\lambda_0}$, we may assume that
$|X-Y|\le \rho$.
Note that $H^\rho=0$ on $[-1/2, 1/2]^d\setminus [-3/8,3/8]^d$. 
Thus it is enough to consider the case where
$X, Y\in [-1/2,1/2]^d$.
We may further assume that $\delta(X)\le \rho$ or $\delta (Y)\le \rho$.
For otherwise, $H^\rho(X)=H^\rho (Y)=0$ and there is nothing to show.
Finally, suppose that $\delta(Y)\le \rho$. Then
$$
\aligned
|H^\rho(X)-H^\rho (Y)|
 & \le 
\theta (\rho^{-1} \delta(X)) |\big(A(X)-\bar{A}(X)) -( A(Y)-\bar{A}(Y)\big)|\\
& \qquad\qquad +|A(Y)-\bar{A}(Y)| |\theta (\rho^{-1} \delta(X)) -\theta (\rho^{-1} \delta (Y))|\\
& 
\le C|X-Y|^{\lambda_0}
+ C \big\{ \delta (Y)\big\}^{\lambda_0} |X-Y|\cdot \rho^{-1}\\
& \le C|X-Y|^{\lambda_0} +C\rho^{\lambda_0-1} |X-Y|\\
&\le C |X-Y|^{\lambda_0}.
\endaligned
$$
The proof for the case $\delta(X)\le \rho$ is the same.
\end{proof}

It follows from Lemma \ref{lemma-6.2} that for $\rho\in (0,1/4)$,
\begin{equation}\label{I-7.2}
\| A^\rho -\bar{A}\|_{C^{\lambda_0/2}(\mathbb{R}^d)}
\le C\rho^{\lambda_0/2}.
\end{equation}
Since $A^\rho=A=\bar{A}$ on $\partial\Omega$, we may deduce from 
Theorem \ref{operator-approximation-theorem} that
\begin{equation}\label{I-6.3}
\| \mathcal{K}_{A^\rho} -\mathcal{K}_{\bar{A}}\|_{L^2\to L^2}
\le C \| A^\rho-\bar{A}\|_{C^{\lambda_0/2}(\mathbb{R}^d)}
\le C \rho^{\lambda_0/2}.
\end{equation}
for any $\rho\in (0,1/4)$.
Note that by Lemma \ref{construction-A-bar-lemma} and Theorem \ref{step-one-theorem},
 the operator $(1/2)I+\mathcal{K}_{\bar{A}}$
is invertible on $L^2_0(\partial\Omega, \mathbb{R}^m)$ and $\|((1/2)I +\mathcal{K}_{\bar{A}})^{-1}\|
_{L_0^2\to L_0^2}\le C$.
Write $$
(1/2)I +\mathcal{K}_{A^\rho} =(1/2)I +\mathcal{K}_{\bar{A}}
+(\mathcal{K}_{A^\rho}-\mathcal{K}_{\bar{A}}).
$$
In view of (\ref{I-6.3}), one may choose $\rho>0$ depending only on $\mu$, $\lambda$,
$\tau$ and the Lipschitz character of $\Omega$ so that
$$
\| ((1/2)I +\mathcal{K}_{\bar{A}})^{-1} (\mathcal{K}_{A^\rho}-\mathcal{K}_{\bar{A}})\|
_{L_0^2\to L_0^2} \le 1/2.
$$
It follows that
$(1/2)I+\mathcal{K}_{A^\rho}$
is invertible on $L^2_0(\partial\Omega, \mathbb{R}^m)$ and 
$$
\|((1/2)I +\mathcal{K}_{A^\rho})^{-1}\|
_{L^2_0\to L^2_0}\le 2
\|((1/2)I +\mathcal{K}_{\bar{A}})^{-1}\|
_{L^2_0\to L_0^2}\le 2C.
$$
Similar arguments show that  it is possible to choose $\rho$ depending only on
$\mu$, $\lambda$, $\tau$ and the Lipschitz character of $\Omega$
such that $-(1/2)I +\mathcal{K}^*_\rho: L^2(\partial\Omega, \mathbb{R}^m)
\to
L^2(\partial\Omega, \mathbb{R}^m)
$
and $\mathcal{S}_{A^\rho}: L^2(\partial\Omega, \mathbb{R}^m) \to
W^{1,2} (\partial\Omega, \mathbb{R}^m)$
are invertible and the operator norms of their inverses are bounded
by a ``good'' constant.
Let $\widetilde{A}=A^\rho$.
Note that if $\text{dist}(X, \partial\Omega)\le (1/4)\rho$,
$A^\rho(X)=A(X)$. 
This completes Step Two.

\bigskip

Step Three, which is given in the following lemma,
 involves a perturbation argument.

\begin{lemma}\label{perturbation-lemma}
Let $A^0=(a_{ij}^{\alpha\beta}), A^1
=(b_{ij}^{\alpha\beta})\in \Lambda(\mu, \lambda, \tau)$.
Let $\Omega$ be a bounded Lipschitz domain.
Suppose that $A^0=A^1$ in $\{ X\in \Omega: \text{dist}(X,\partial\Omega)
\le c_0r_0\}$ for some $c_0>0$, where $r_0=\text{diam}(\Omega)$.
Assume that $\mathcal{L}^0 =-\text{div}(A^0\nabla )$ has the Rellich property
in $\Omega$ with constant $C_0$.
Then $\mathcal{L}^1=-\text{div}(A^1\nabla)$
has the Rellich property in $\Omega$
with constant $C_1$, where $C_1$ depends only on $d$, $m$,
$\mu$, $\lambda$, $\tau$, $c_0$, $C_0$ and the Lipschitz character of $\Omega$.
\end{lemma}

\begin{proof}
Suppose that $\mathcal{L}^1 (u)=0$ in $\Omega$, $(\nabla u)^*\in L^2(\partial\Omega)$
and $\nabla u$ exists n.t. on $\partial\Omega$.
Let $\varphi\in C_0^\infty(\mathbb{R}^d)$ such that $|\nabla \varphi|\le Cr_0^{-1}$,
$\varphi=1$ on $\{ X\in \mathbb{R}^d: \text{dist}(X,\partial\Omega)\le (1/4)c_0r_0\}$
and
$\varphi=0$ on $\{ X\in \mathbb{R}^d: \text{dist}(X,\partial\Omega)\ge (1/2)c_0r_0\}$.
Let $\bar{u}=\varphi(u-E)$, where $E=u_\Omega$ is the average of $u$ over $\Omega$.
Note that 
$$
\mathcal{L}^0 (\bar{u})
=-\partial_i \big\{ a_{ij}^{\alpha\beta} (\partial_j \varphi) (u-E)^\beta\big\}
-a_{ij}^{\alpha\beta} (\partial_i\varphi)(\partial_j u),
$$
where we have used the fact that $\mathcal{L}^0(u)=\mathcal{L}^1 (u)=0$ on 
$\{ X\in \Omega: \text{dist}(X,\partial\Omega)<c_0r_0\}$.
It follows from (\ref{Green-representation-formula-1}) that
\begin{equation}\label{I-6.7}
\bar{u}(X)=\mathcal{S}_{A^0}
\left(\frac{\partial \bar{u}}{\partial \nu_{A^0}}\right)
-\mathcal{D}_{A^0} (\bar{u}) +v(X)=w(X)+ v(X),
\end{equation}
where $v$ satisfies
\begin{equation}\label{I-6.8}
\aligned
|\nabla v(X)|
&\le C \int_\Omega
|\nabla_X\nabla_Y \Gamma_{A^0}(X,Y)||\nabla \varphi||u-E|\,dY\\
& \qquad\qquad 
+C \int_\Omega |\nabla_X \Gamma_{A^0} (X,Y)||\nabla\varphi| |\nabla u|\, dY.
\endaligned
\end{equation}
This, together with (\ref{size-estimate}) and 
(\ref{2-derivative-estimate}), implies that
if $X\in \Omega$ and $\text{dist}(X, \partial\Omega)
\le (1/5)c_0r_0$,
\begin{equation}\label{estimate-of-v}
|\nabla v(X)|^2
  \le  \frac{C}{r_0^d}
\int_\Omega |\nabla u|^2\, dY
\le Cr_0^{1-d} \|\frac{\partial u}{\partial\nu_{A^1}}\|_2 \|\nabla_{tan} u\|_2,
\end{equation}
where we have used (\ref{Green-estimate}) for the last inequality.

Next, note that
$\mathcal{L}^0 (w)=0$ in $\Omega$, where $w=\bar{u}-v$.
Using (\ref{estimate-of-v}) and the assumption $(\nabla u)^*\in L^2
(\partial\Omega)$ as well as the $L^\infty$ gradient estimate
(\ref{gradient-estimate}), we may deduce that
$(\nabla w)^*\in L^2(\partial\Omega)$ and
$\nabla w$ exists n.t. on $\partial\Omega$.
Since $\mathcal{L}^0$ has the Rellich property, this implies that
\begin{equation}\label{I-6.9}
\aligned
\|\nabla w\|_2  & \le C_0 \|\frac{\partial w}{\partial \nu_{A^0}}\|_2
\le C\big\{  \| \frac{\partial u}{\partial \nu_{A^1}} \|_2 
+\|\nabla v\|_2\big\} \\
&\le C \left\{ 
\| \frac{\partial u}{\partial \nu_{A^1}} \|_2 
+\|\frac{\partial u}{\partial\nu_{A^1}}\|_2^{1/2} \|\nabla_{tan} u\|_2^{1/2}\right\},
\endaligned
\end{equation}
where we used (\ref{estimate-of-v}) in the last inequality.
Using (\ref{estimate-of-v}) again, we obtain
\begin{equation}\label{I-10}
\aligned
\| \nabla u \|_2
& \le C \big\{
\|\nabla w \|_2
+\| \nabla v \|_2\big\}\\
&\le C \left\{
\|\frac{\partial u}{\partial \nu_{A^1}}\|_2
+
\|\frac{\partial u}{\partial\nu_{A^1}}\|_2^{1/2} \|\nabla_{tan} u\|_2^{1/2}
\right\}.
\endaligned
\end{equation}
The desired estimate $\|\nabla u\|_2 \le C\|\frac{\partial u}{\partial \nu_{A^1}}\|_2$
follows readily from (\ref{I-10}) by the Cauchy inequality with an $\varep$.
The proof of $\|\nabla u\|_2 \le C\|\nabla_{tan} u\|_2$ is similar.
\end{proof}

Finally we give the proof of Theorem \ref{Holder-continuous-Rellich}.

\noindent{\bf Proof of Theorem \ref{Holder-continuous-Rellich}}.
By Step Two there exists $\widetilde{A}\in \Lambda (\mu, \lambda_0, \tau_0)$
such that $\widetilde{A}=A$ in $\{ X\in \mathbb{R}^d: \text{dist}(X, \partial\Omega)\le c\}$
and $(1/2)I +\mathcal{K}_{\widetilde{A}}: L^2(\partial\Omega, \mathbb{R}^m)
\to L^2(\partial\Omega, \mathbb{R}^m)$, 
$\mathcal{S}_{\widetilde{A}}: L^2(\partial\Omega, \mathbb{R}^m)\to
W^{1,2} (\partial\Omega, \mathbb{R}^m)$
are invertible.
Moreover, the operator norms of these inverses are bounded by
a ``good'' constant $C$.
It follows that the $L^2$ Neumann and regularity problems for
$\mathcal{L}^{\widetilde{A}} (u)=0$ in $\Omega$
are uniquely solvable
and the solutions satisfy
$\|(\nabla u)^*\|_2 \le C\|\frac{\partial u}{\partial \nu}\|_2$
and $\|(\nabla u)^*\|_2 \le C\|\nabla_{tan} u\|_2$ with ``good''
constant $C$.
In particular the operator $\mathcal{L}^{\widetilde{A}}$
has the Rellich property in $\Omega$ with a ``good'' constant $C$.
By Lemma \ref{perturbation-lemma} this implies that
$\mathcal{L}$ has the Rellich property in $\Omega$ with a ``good'' constant $C$.
The proof is complete.
\qed

\section{Rellich estimates for large scales}

Let $\psi$ be a Lipschitz function on $\mathbb{R}^{d-1}$ such that $\psi (0)=0$ and
$\|\nabla \psi\|_\infty\le M$.
Let $D(r)$ and $\Delta (r)$ be defined as in (\ref{definition-of-Z}).
In this section we establish the following.

\begin{thm}\label{periodic-Rellich-1}
Let $\mathcal{L}=-\text{div}(A\nabla)$
with $A\in \Lambda(\mu, \lambda,\tau)$ and $A^*=A$.
Assume that $A\in C^2 (\mathbb{R}^d)$. Suppose that 
$\mathcal{L}(u)=0$ in $D(8r)$ for some $r>0$,
where $u\in C^2({D(8r)})$, 
$(\nabla u)^*_{D(8r)} \in L^2 (\Delta (6r))$ and $\nabla u$ exists n.t. on $\Delta (6r)$.
Then
\begin{equation}\label{periodic-Rellich-estimate}
\aligned
\int_{\Delta (r)} |\nabla u|^2\, d\sigma
 & \le C\int_{\Delta(4r)} |\frac{\partial u}{\partial \nu} |^2\, d\sigma
+\frac{C}{r} \int_{D(4r)}
|\nabla u|^2\, dX,\\
\int_{\Delta (r)} |\nabla u|^2\, d\sigma
 & \le C\int_{\Delta(4r)} |\nabla_{tan} u |^2\, d\sigma
+\frac{C}{r} \int_{D(4r)}
|\nabla u|^2\, dX,
\endaligned
\end{equation}
where $C$ depends only on $\mu$, $\lambda$, $\tau$ and $M$.
\end{thm}

Observe that by Theorem \ref{Holder-continuous-Rellich} and the localization
techniques used in the proof of Lemma \ref{interior-exterior-Rellich},
the estimates in (\ref{periodic-Rellich-estimate}) hold
for $0<r\le 3$ with $C=C(\mu,\lambda, \tau, M)>0$. 
We will rely on the Rellich identities developed in \cite{kenig-shen-1}
treat the case $r>3$.

Let $Q$ be the difference operator defined by
\begin{equation}\label{definition-of-Q}
Q(f)(x^\prime,x_d)=f(x^\prime,x_d+1)-f(x^\prime,x_d).
\end{equation}
It is easy to verify that 
\begin{equation}\label{product-rule}
Q(fg)-Q(f)Q(g)=fQ(g)+g Q(f)
\end{equation}
and
\begin{equation}\label{property-of-Q}
\int_{D(r)} Q(f)\, dX
=\int_{\substack
{ |x^\prime| <r\\  Cr<  x_d<Cr+1 } }
f\, dX
-\int_{\substack
{|x^\prime|<r\\ \psi(x^\prime)<x_d<\psi(x^\prime)+1
}} 
f\, dX,
\end{equation}
where $C=10 \sqrt{d} (M+1)$.
Also note that since
$A(x^\prime,x_d+1)=A(x^\prime,x_d)$, we have $Q\mathcal{L}=
\mathcal{L}Q$.
In particular, $\mathcal{L}(Q(u))=0$ whenever $\mathcal{L}(u)=0$.

\begin{lemma}\label{Neumann-Rellich} Let $r>1$.
Under the same conditions on $A$ and $u$ as in Theorem \ref{periodic-Rellich-1}, we have
\begin{equation}\label{Neumann-Rellich-estimate}
\int_{\substack{|x^\prime|<r\\
\psi(x^\prime)<x_d <\psi(x^\prime)+1
}}
|\nabla u|^2\, dX
\le C\int_{\Delta(2r)}
\big|\frac{\partial u}{\partial\nu}\big|^2\, d\sigma
+\frac{C}{r}
\int_{D(3r)} |\nabla u|^2\, dX,
\end{equation}
where $C$ depends only on $\mu$ and $M$.
\end{lemma}

\begin{proof} 
By approximating the domain $D(3r)$ from inside, 
 we may assume that $u\in C^2(\overline{D(3r)})$.
Let $\Omega_\rho = D(\rho)$ for $\rho\in (r,2r)$.
 It follows from integration by parts
that
$$
\aligned
\int_{\partial\Omega_\rho} \frac{\partial u}{\partial\nu} \cdot Q(u)\, d\sigma
& =
\int_{\Omega_\rho}
a_{ij}^{\alpha\beta} \, \frac{\partial u^\beta}{\partial x_j} 
\cdot Q\left(\frac{\partial u^\alpha}{\partial x_i}\right)\, dX\\
& =\frac12
\int_{\Omega_\rho}
a_{ij}^{\alpha\beta}
\left\{ \frac{\partial u^\beta}{\partial x_j} \cdot Q\left(\frac{\partial u^\alpha}{\partial x_i}\right)
+
Q\left(\frac{\partial u^\beta}{\partial x_j} \right)
\cdot \frac{\partial u^\alpha}{\partial x_i}\right\}\, dX\\
&=\frac12
\int_{\Omega_\rho}
a_{ij}^{\alpha\beta}
\left\{ Q\left(\frac{\partial u^\beta}
{\partial x_j}\cdot \frac{\partial u^\alpha}{\partial x_i}\right)
-
Q\left(\frac{\partial u^\beta}{\partial x_j}\right) 
\cdot Q\left(\frac{\partial u^\alpha}{\partial x_i}\right)\right\}dX\\
&\le 
\frac12
\int_{\Omega_\rho}
a_{ij}^{\alpha\beta}\, 
Q\left(\frac{\partial u^\beta}{\partial x_j}
\cdot \frac{\partial u^\alpha}{\partial x_i}\right)\, dX,
\endaligned
$$
where we have used the symmetry condition (\ref{symmetry}), (\ref{product-rule}) 
and the ellipticity condition (\ref{ellipticity}).
This, together with the periodicity condition on $A(X)$, gives
$$
-\int_{\Omega_\rho}
Q\left( a_{ij}^{\alpha\beta}
\frac{\partial u^\beta}{\partial x_j}\cdot \frac{\partial u^\alpha}{\partial x_i}\right)\, dX
\le -2 \int_{\partial\Omega_\rho}
\frac{\partial u}{\partial \nu}
\cdot Q(u)\, d\sigma.
$$
In view of (\ref{property-of-Q}), we obtain
\begin{equation}\label{Neumann-Rellich-1}
\aligned
& \mu\int_{\substack{
|x^\prime|<\rho\\
\psi(x)<x_d<\psi(x)+1
}}
|\nabla u|^2\, dX\\
&   \le -2\int_{\partial\Omega_\rho}
\frac{\partial u}{\partial \nu} \cdot Q(u)\, d\sigma
+\frac{1}{\mu}\int_{\substack{
|x^\prime|<\rho\\
C\rho<x_d<C\rho+1
}}
 |\nabla u|^2\, dX\\
&\le \delta \int_{\Delta(\rho)} |Q(u)|^2\, d\sigma
+\frac{1}{\delta} \int_{\Delta(\rho)} |\frac{\partial u}{\partial \nu}|^2\,d \sigma\\
&\quad\quad  +C\int_{\partial\Omega_\rho\setminus \Delta(\rho)} |\nabla u| |Q(u)|\, d\sigma
+C \int_{\substack{
|x^\prime|<\rho\\
C\rho<x_d<C\rho+1
}}|\nabla u|^2\, dX
\endaligned
\end{equation}
for $\rho\in (r,2r)$,
 where $\delta\in (0,1)$ and we have used the Cauchy inequality.

Next, using $|Q(u)(x^\prime,x_d)|^2 
\le \int_{x_d}^{x_d+1} |\frac{\partial u}{\partial s}(x^\prime,s)|^2\, ds$, we see that
\begin{equation}\label{Neumann-Rellich-1.5}
\int_{\Delta(\rho)} |Q(u)|^2\, d\sigma
\le C \int_{\substack{
|x^\prime|<\rho \\
\psi (x^\prime)< x_d < \psi(x^\prime) +1
}}
|\nabla u|^2\, dX
\end{equation}
and
\begin{equation}\label{Neumann-Rellich-2}
\aligned
& \int_{\partial\Omega_\rho \setminus \Delta(\rho)} |Q(u)|^2\, d\sigma\\
&\quad\le C \int_{\partial B(0,\rho)} \int_{\psi(x^\prime)}^{\psi(x^\prime)+\rho+1}
|\nabla u (x^\prime, s)|^2\, ds d\sigma
+C \int_{\substack{
|x^\prime|<\rho\\
C\rho<x_d <C\rho +1
}}
 |\nabla u|^2\, dX,
\endaligned
\end{equation}
where $C$ depends only on $M$.

Finally we choose $\delta>0$ in (\ref{Neumann-Rellich-1})
so small that $\delta C\le (1/2)\mu$.
In view of (\ref{Neumann-Rellich-1}), (\ref{Neumann-Rellich-1.5}) and (\ref{Neumann-Rellich-2}),
we obtain
\begin{equation}\label{Neumann-Rellich-3}
\aligned
& \int_{\substack{
|x^\prime|<r\\
\psi (x^\prime)< x_d< \psi (x^\prime)+1
}}
 |\nabla u|^2\, dX \\
&\le  C\int_{\Delta(2r)}
\big|\frac{\partial u}{\partial \nu}\big|^2\, d\sigma
+
C \int_{\partial B(0,\rho)} \int_{\psi(x^\prime)}^{\psi(x^\prime)+2r+1}
|\nabla u (x^\prime, s)|^2\, ds d\sigma\\
& \qquad
+C\int_{|x^\prime|<2r} |\nabla u(x^\prime, C\rho)|^2\, dx^\prime
+C \int_{\substack{
|x^\prime|<2r\\
C\rho<x_d <C\rho +1
}}
 |\nabla u|^2\, dX,
\endaligned
\end{equation}
The desired estimate follows by integrating both sides of (\ref{Neumann-Rellich-3})
with respect to $\rho$ over the interval $(r,2r)$.
\end{proof}

\begin{lemma}\label{regularity-Rellich} Let $r>1$.
Under the same conditions on $A$ and $u$, we have
\begin{equation}\label{regularity-Rellich-estimate}
\int_{\substack{
 |x^\prime|<r\\
\psi(x^\prime)<x_d <\psi(x^\prime)+1
}}
|\nabla u|^2\, dX
\le C\int_{\Delta(2r)}
\big|\nabla_{tan} u|^2\, d\sigma
+\frac{C}{r}
\int_{D(3r)} |\nabla u|^2\, dX,
\end{equation}
where $C$ depends only on $\mu$ and $M$.
\end{lemma}

\begin{proof}
Let $\Omega_\rho =D(\rho)$ for $\rho\in (r,2r)$. 
As in the proof of Lemma \ref{Neumann-Rellich},
 we have
\begin{equation}\label{regularity-Rellich-1}
\aligned
& \mu\int_{\substack{
|x^\prime|<\rho\\
\psi(x^\prime)<x_d<\psi(x^\prime)+1
}}
|\nabla u|^2\, dtdx\\
&   \le -2\int_{\partial\Omega_\rho}
\frac{\partial u}{\partial \nu} \cdot Q(u)\, d\sigma
+\frac{1}{\mu}\int_{\substack{
|x^\prime|<\rho\\
C\rho< x_d<C\rho +1
}}
|\nabla u|^2\, dX.
\endaligned
\end{equation}
To estimate the first term in the right-hand side of (\ref{regularity-Rellich-1}), we 
observe that
\begin{equation}\label{regularity-Rellich-2}
\aligned
& \int_{\partial\Omega_\rho} \frac{\partial u}{\partial \nu}\cdot Q(u)\, d\sigma
=\int_{\partial \Omega_\rho}
u\cdot \frac{\partial}{\partial\nu} \big\{ Q(u)\big\} \, d\sigma
=\int_{\partial \Omega_\rho}
u^\alpha \cdot n_i a_{ij}^{\alpha\beta} 
\frac{\partial }{\partial x_j} Q(u^\beta)\, d\sigma\\
&=\int_{\partial\Omega_\rho}
u^\alpha \cdot n_i \left\{ 
\int_0^1 \frac{\partial }{\partial s}
\left( a_{ij}^{\alpha\beta} (x^\prime,x_d+s) 
\frac{\partial u^\beta}{\partial x_j} (x^\prime, x_d+s)\right)\, ds\right\} d\sigma \\
&=\int_{\partial\Omega_\rho}
u^\alpha\cdot n_i 
\frac{\partial }{\partial x_d}
\left\{ \int_0^1 a_{ij}^{\alpha\beta} (x^\prime,x_d+s)
 \frac{\partial u^\beta}{\partial x_j}(x^\prime,x_d+s)\, ds \right\}
d\sigma\\
&=\int_{\partial\Omega_\rho}
u^\alpha\left( n_i\frac{\partial}{\partial x_d}
-n_{d} \frac{\partial}{\partial x_i}\right) 
\left\{ \int_0^1 a_{ij}^{\alpha\beta} (x^\prime,x_d+s) 
\frac{\partial u^\beta}{\partial x_j}(x^\prime,x_d+s)\, ds \right\}
d\sigma\\
&
=-\int_{\partial\Omega_\rho}
\left(n_i \frac{\partial}{\partial x_d}
-n_{d} \frac{\partial}{\partial x_i}\right)u^\alpha
\cdot 
\left\{ \int_0^1 a_{ij}^{\alpha\beta} (x,x_d+s) 
\frac{\partial u^\beta}{\partial x_j}(x,x_d+s)\, ds \right\}
d\sigma,
\endaligned
\end{equation}
where we have used the facts that $\mathcal{L}(Q(u))=0$ and
that $(n_i \frac{\partial}{\partial x_d} -n_{d} \frac{\partial}{\partial x_i})$ is 
a tangential derivative.
It follows that
\begin{equation}\label{regularity-Rellich-3}
\aligned
& \big|\int_{\partial\Omega_\rho} \frac{\partial u}{\partial \nu}\cdot Q(u)\, d\sigma\big|\\
& \le C\int_{\partial\Omega_\rho}
|\nabla_{tan}u (X) |\left\{ \int_0^1 |\nabla u(x^\prime,x_d+s)|ds\right\} d\sigma\\
&\le \delta \int_{\partial\Omega_\rho} \left\{ 
\int_0^1 |\nabla u(x^\prime,x_d+s)|^2\, ds\right\} d\sigma
+C_\delta \int_{\partial\Omega_\rho} |\nabla_{tan} u|^2\, d\sigma\\
&\le C\delta\int_{
\substack{
|x^\prime|<\rho\\
\psi(x^\prime)< x_d<\psi(x^\prime) +1
}}
|\nabla u|^2\, dX
+C_\delta \int_{\Delta(\rho)} |\nabla_{tan} u|^2\, d\sigma\\
& \quad\quad +C\int_{\partial\Omega_\rho\setminus \Delta(\rho)} |\nabla u|^2\, d\sigma
+C \int_{\partial\Omega_\rho\setminus \Delta(\rho)}
\left\{\int_0^1 |\nabla u(x^\prime,x_d+s)|^2\, ds \right\} d\sigma,
\endaligned
\end{equation}
where $\delta\in (0,1)$ and we have used the Cauchy inequality.

We now choose $\delta$ so that $C\delta<(1/4)\mu$.
 In view of (\ref{regularity-Rellich-1})
and (\ref{regularity-Rellich-3}), we obtain
\begin{equation}\label{regularity-Rellich-4}
\aligned
& 
\int_{
\substack{
|x^\prime|<r\\
\psi(x^\prime)< x_d<\psi(x^\prime) +1
}}
|\nabla u|^2\, dX
\\
& \le C\int_{\Delta(2r)} |\nabla_{tan} u|^2\, d\sigma
+C
\int_{
\substack{
|x^\prime|<2r\\
C\rho< x_d<C\rho +1
}}
|\nabla u|^2\, dX\\
& \quad\quad +C\int_{\partial\Omega_\rho\setminus \Delta(\rho)} |\nabla u|^2\, d\sigma
+C \int_{\partial\Omega_\rho\setminus \Delta(\rho)}
\left\{\int_0^1 |\nabla u(x^\prime,x_d+s)|^2\, ds \right\} d\sigma,
\endaligned
\end{equation}
for any $\rho\in (r,2r)$. Estimate (\ref{regularity-Rellich-estimate})
follows by integrating both sides of (\ref{regularity-Rellich-4})
with respect to $\rho$ over the interval $(r,2r)$.
\end{proof}

We now give the proof of Theorem \ref{periodic-Rellich-1}.

\noindent{\bf Proof of Theorem \ref{periodic-Rellich-1}.}
We may assume that $r>3$.
By covering $\Delta (r)$ with surface balls of small radius $c(M)$ on $\{ (x^\prime, \psi(x^\prime)): 
x^\prime\in \mathbb{R}^{d-1}\}$
and using the first inequality 
in (\ref{periodic-Rellich-estimate}) on each small surface ball, we obtain
\begin{equation}\label{p-R-1}
\int_{\Delta (r)}
|\nabla u|^2\, dX
\le C\int_{\Delta (2r)} \big|\frac{\partial u}{\partial \nu}\big|^2\, d\sigma
+ C\int_{\substack{
|x^\prime|<r+1\\
\psi (x^\prime)<x_d<\psi (x^\prime)+1
}} |\nabla u|^2\, dX.
\end{equation}
This, together with Lemma \ref{Neumann-Rellich}, gives
the first inequality in (\ref{periodic-Rellich-estimate}).
The second inequality in (\ref{periodic-Rellich-estimate})
follows from Lemma \ref{regularity-Rellich}
in a similar fashion.
\qed

\section{Proof of Theorems \ref{e-Dirichlet-theorem}, \ref{e-regularity-theorem}
and \ref{e-Neumann-theorem}}

Let $A\in \Lambda(\mu, \lambda, \tau)$ with $A^*=A$.
As we pointed out in the Introduction, by a simple rescaling argument, 
it suffices to prove
Theorems \ref{e-Dirichlet-theorem}, \ref{e-regularity-theorem}
and \ref{e-Neumann-theorem} for $\varep=1$, but with constant $C$
depending only on $d$, $m$, $\mu$, $\lambda$, $\tau$ and
the Lipschitz character of $\Omega$.
Moreover, note that the existence and uniqueness as well as representations by layer potentials 
of solutions to
the $L^2$ Dirichlet, regularity and
Neumann problems for $\mathcal{L}(u)=0$ in $\Omega$ 
were already given by Theorem \ref{local-solvability} and its proof.
As a result,
we only need to show that the operator norms of 
$((1/2)I+\mathcal{K}_A)^{-1}$ on $L_0^2(\partial\Omega, \mathbb{R}^m)$,
$(-(1/2)I +\mathcal{K}_A)^{-1}$ on $L^2(\partial\Omega, \mathbb{R}^m)$
and $\mathcal{S}_A^{-1}: W^{1,2}(\partial\Omega, \mathbb{R}^m)
\to L^2(\partial\Omega, \mathbb{R}^m)$
are bounded by a ``good'' constant.

To this end, 
for any $A\in \Lambda (\mu, \lambda, \tau)$, we choose
a sequence  $\{ A^k\}\subset\Lambda (\mu, \lambda/2, \eta)$, where
$\eta=\eta (\mu, \lambda, \tau)$, such that $A^k\in C^2(\mathbb{R}^d)$ and
$\| A^k -A\|_{C^{\lambda/2}(\mathbb{R}^d)} \to 0$.
By Theorem \ref{Rellich-imply-solvability} as well as its proof,
it follows from Theorem \ref{periodic-Rellich-1} by a simple
localization argument that $\mathcal{L}^k =-\text{div}(A^k\nabla)$
has the Rellich property in any Lipschitz
domain $\Omega$
 with constant $C(\Omega)$ depending only on $d$, $m$, $\mu$, $\lambda$, $\tau$
and the Lispchitz character of $\Omega$.
This implies that
the operator norms of $(\pm (1/2)I+\mathcal{K}_{A^k})^{-1}$ and
$\mathcal{S}_{A^k}^{-1}$ are bounded by a ``good'' constant $C_0$.
Since $\| \mathcal{K}_{A^k}-\mathcal{K}_A\|_{L^2\to L^2}
\to 0$ and $\| \mathcal{S}_{A^k}-\mathcal{S}_A\|_{L^2\to W^{1,2}} \to 0$
by Theorem \ref{operator-approximation-theorem}, we may conclude that
the operator norms of $(\pm (1/2)I+\mathcal{K}_{A})^{-1}$ and
$\mathcal{S}_{A}^{-1}$ are bounded by the same ``good'' constant $C_0$.
This completes the proof
of Theorems 
\ref{e-Dirichlet-theorem}, \ref{e-regularity-theorem}
and \ref{e-Neumann-theorem}.
\qed

\bibliography{ks2}

\small
\noindent\textsc{Department of Mathematics, 
University of Chicago, Chicago, IL 60637}\\
\emph{E-mail address}: \texttt{cek@math.uchicago.edu} \\

\noindent\textsc{Department of Mathematics, 
University of Kentucky, Lexington, KY 40506}\\
\emph{E-mail address}: \texttt{zshen2@email.uky.edu} \\

\noindent \today

\end{document}